%% file: non-archimedean_Gromov_compactness_arXiv.tex
\pdfoutput=1
\newif\ifpersonal
\RequirePackage[l2tabu, orthodox]{nag} 
\documentclass[12pt,a4paper,reqno]{amsart} 
\usepackage{amssymb,amsfonts,amsthm,amsmath,mathrsfs,mathtools}
\usepackage{microtype,fixltx2e}
\usepackage{enumerate,comment,braket,tikz-cd,xspace}
\usepackage[utf8]{inputenc}
\usepackage[T1]{fontenc}
\usepackage[hidelinks]{hyperref}
\usepackage[capitalize]{cleveref}
\usepackage[centering,vscale=0.7,hscale=0.7]{geometry}

\theoremstyle{plain}
\newtheorem{thm}{Theorem}[section]
\newtheorem{lem}[thm]{Lemma}
\newtheorem{prop}[thm]{Proposition}
\newtheorem{cor}[thm]{Corollary}
\newtheorem{conj}[thm]{Conjecture}
\theoremstyle{definition}
\newtheorem{defin}[thm]{Definition}
\theoremstyle{remark}

\newtheorem{rem}[thm]{Remark}
\numberwithin{equation}{section}

\input{newcommands_all}

\begin{document}
\title{Gromov compactness in non-archimedean analytic geometry}
\author{Tony Yue YU}
\address{Tony Yue YU, Institut de Mathématiques de Jussieu - Paris Rive Gauche, CNRS-UMR 7586, Case 7012, Université Paris Diderot - Paris 7, Bâtiment Sophie Germain 75205 Paris Cedex 13 France}
\email{yuyuetony@gmail.com}
\date{July 30, 2014 (revised on June 30, 2015)}

\begin{abstract}
Gromov's compactness theorem for pseudo-holomorphic curves is a foundational result in symplectic geometry.
It controls the compactness of the moduli space of pseudo-holomorphic curves with bounded area in a symplectic manifold.
In this paper, we prove the analog of Gromov's compactness theorem in non-archimedean analytic geometry.
We work in the framework of Berkovich spaces.
First, we introduce a notion of Kähler structure in non-archimedean analytic geometry using metrizations of virtual line bundles.
Second, we introduce formal stacks and non-archimedean analytic stacks.
Then we construct the moduli stack of non-archimedean analytic stable maps using formal models, Artin's representability criterion and the geometry of stable curves.
Finally, we reduce the non-archimedean problem to the known compactness results in algebraic geometry.
The motivation of this paper is to provide the foundations for non-archimedean enumerative geometry.
\end{abstract}

\maketitle

\personal{PERSONAL COMMENTS ARE SHOWN!!!}

\tableofcontents

\input{non-archimedean_Gromov_compactness_body}
\bibliographystyle{plain}
\bibliography{dahema}

\end{document}

%% file: newcommands_all.tex
\ifpersonal
\newcommand*{\personal}[1]{\textcolor{blue}{(Personal: #1)}}
\newcommand*{\todo}[1]{\textcolor{red}{(Todo: #1)}}
\else
\newcommand*{\personal}[1]{\ignorespaces}
\newcommand*{\todo}[1]{\ignorespaces}
\fi

\newcommand{\R}{\mathbb R}
\newcommand{\Z}{\mathbb Z}

\newcommand{\fA}{\mathfrak A}

\newcommand{\fC}{\mathfrak C}

\newcommand{\fS}{\mathfrak S}
\newcommand{\fT}{\mathfrak T}
\newcommand{\fU}{\mathfrak U}

\newcommand{\fX}{\mathfrak X}
\newcommand{\fY}{\mathfrak Y}

\newcommand{\ff}{\mathfrak f}

\newcommand{\fs}{\mathfrak s}

\newcommand{\cA}{\mathcal A}
\newcommand{\cC}{\mathcal C}
\newcommand{\cD}{\mathcal D}

\newcommand{\cI}{\mathcal I}

\newcommand{\cM}{\mathcal M}

\newcommand{\cS}{\mathcal S}

\newcommand{\cX}{\mathcal X}

\newcommand{\bP}{\mathbf P}

\newcommand{\hL}{\widehat L}

\newcommand{\tD}{\widetilde D}

\newcommand{\IX}{I_\fX}
\newcommand{\IY}{I_\fY}

\newcommand{\SX}{S_\fX}

\newcommand{\fXe}{\fX_\eta}
\newcommand{\fXs}{\fX_s}

\newcommand{\fYe}{\fY_\eta}
\newcommand{\fYs}{\fY_s}

\newcommand{\Sch}{\mathrm{Sch}}
\newcommand{\FSch}{\mathrm{FSch}}

\newcommand{\An}{\mathrm{An}}

\newcommand{\bcMgn}{\overline{\mathcal M}_{g,n}}

\newcommand{\bcMgnprime}{\overline{\mathcal M}_{g,n'}}
\newcommand{\bcMgnijprime}{\overline{\mathcal M}_{g_{ij},n'_{ij}}}

\newcommand{\Simp}{\mathit{Simp}}

\newcommand{\kc}{k^\circ}

\newcommand{\an}{^\mathrm{an}}
\newcommand{\alg}{^\mathrm{alg}}

\newcommand{\mult}{\mathit{mult}}
\newcommand{\inv}{^{-1}}

\newcommand{\gn}{$n$-pointed genus $g$ }
\newcommand{\gnprime}{$n'$-pointed genus $g$ }

\newcommand{\kanal}{$k$-analytic\xspace}

\newcommand{\red}{^\mathrm{red}}

\usetikzlibrary{decorations.markings} 
\tikzset{
  closed/.style = {decoration = {markings, mark = at position 0.5 with { \node[transform shape, xscale = .8, yscale=.4] {/}; } }, postaction = {decorate} },
  open/.style = {decoration = {markings, mark = at position 0.5 with { \node[transform shape, scale = .7] {$\circ$}; } }, postaction = {decorate} }
}

\DeclareMathOperator{\Aut}{Aut}

\DeclareMathOperator{\Div}{Div}

\DeclareMathOperator{\Ext}{Ext}

\DeclareMathOperator{\Hom}{Hom}
\DeclareMathOperator{\Image}{Im}
\DeclareMathOperator{\Int}{Int}
\DeclareMathOperator{\Isom}{Isom}
\DeclareMathOperator{\Ker}{Ker}

\DeclareMathOperator{\Mor}{Mor}
\DeclareMathOperator{\NE}{NE}

\DeclareMathOperator{\Proj}{Proj}

\DeclareMathOperator{\SpB}{Sp_\mathrm{B}}
\DeclareMathOperator{\Spec}{Spec}
\DeclareMathOperator{\Spf}{Spf}

\DeclareMathOperator{\val}{val}

%% file: non-archimedean_Gromov_compactness_body.tex
\section{Introduction}\label{sec:intro_(Gromov)}

Counting curves in algebraic varieties is an old subject. It gives non-trivial invariants, displays rich geometrical phenomena, and is intimately related to string theory in theoretical physics.
However, counting the number of curves naively rarely gives satisfactory answers.
It is not clear whether such numbers are finite, or whether they possess expected properties.
To do it in a scientific way, the first step is to compactify the moduli space of all possible curves in question.
In the context of symplectic geometry, Mikhail Gromov discovered that if we consider parametrized curves and allow double point singularities on the source curve, then the moduli space becomes compact.

\begin{thm}[Symplectic Gromov compactness \cite{Gromov_Pseudoholomorphic_1985,Pansu_Compactness_1994,Kontsevich_Enumeration_1995,Ye_Gromov_1994,Hummel_Gromov_1997}]
Let $X$ be a symplectic manifold equipped with a tame almost complex structure and let $K\subset X$ be a compact subset.
The moduli space of stable maps\footnote{A stable map is a map from a closed curve with at worst double point singularities to the target space $X$ satisfying the stability condition (see Definition \ref{def:stable_maps_(Gromov)}).} into $K$ with uniformly bounded genus and area is compact and Hausdorff.
\end{thm}

The algebraic version of the theorem above was proposed by Maxim Kontsevich.

\begin{thm}[Algebraic Gromov compactness \cite{Kontsevich_Enumeration_1995,Fulton_Stable_1997}] 
Let $X$ be a projective scheme of finite type over a field. The moduli stack of stable maps into $X$ with uniformly bounded genus and degree is a proper algebraic stack.
\end{thm}

In this paper, we prove the analog of Gromov's compactness theorem in non-archimedean analytic geometry\footnote{We work in the framework of Berkovich spaces \cite{Berkovich_Spectral_1990,Berkovich_Etale_1993}.}.

Let us state our main result while precise definitions are given in the body of the paper.

Fix a complete discrete valuation field $k$, a positive real number $A$, and two non-negative integers $g$ and $n$.

\begin{thm}[Non-archimedean Gromov compactness (Theorem \ref{thm:non-archimedean_Gromov})]
Let $X$ be a $k$-analytic space, and let $\hL$ be a Kähler structure on $X$ with respect to an SNC formal model $\fX$ of $X$.
Let $\bcMgn(X,A)$ denote the moduli stack of \gn $k$-analytic stable maps into $X$ whose degree with respect to $\hL$ is bounded by $A$.
Then $\bcMgn(X,A)$ is a compact $k$-analytic stack.
If we assume moreover that the $k$-analytic space $X$ is proper\footnote{Recall that in Berkovich geometry, proper means compact and without boundary.} and that the residue field $\tilde k$ has characteristic zero, then $\bcMgn(X,A)$ is a proper $k$-analytic stack.
\end{thm}

\medskip
The following three theorems are among the main ingredients in the proof of the non-archimedean Gromov compactness theorem above.

\begin{thm}[Construction of the moduli stack of algebraic stable maps in full generality (\cref{thm:algebraicity_of_stack_of_stable_maps})]
Let $\cX$ be a scheme locally of finite presentation over a locally noetherian base scheme $\cS$.
The moduli stack $\bcMgn(\cX/\cS)$ of \gn stable maps into $\cX$ over $\cS$ is an algebraic stack locally of finite presentation over $\cS$.
\end{thm}

\begin{thm}[Formal model of \kanal stable maps (\cref{thm:formal_model_stable_maps})]
Let $\fX$ be a formal scheme locally of finite presentation over the ring of integers $k^\circ$.
Let $T$ be a strictly $k$-affinoid space and let $\big(C\rightarrow T,(s_i),f\big)$ be an \gn \kanal stable map into $\fXe$ over $T$.
Up to passing to a quasi-étale covering of $T$, there exists a formal model $\fT$ of $T$ and an \gn formal stable map into $\fX$ over $\fT$ that gives back the $k$-analytic stable map when we apply the generic fiber functor.
\end{thm}

\begin{thm}[Construction of the moduli stack of \kanal stable maps (\cref{cor:moduli_of_analytic_stable_maps})]
Let $X$ be a paracompact strictly \kanal space.
The moduli stack $\bcMgn(X)$ of \gn \kanal stable maps into $X$ is a \kanal stack.
\end{thm}

\medskip

The motivation of studying Gromov compactness in non-archimedean analytic geometry comes from mirror symmetry.
Mirror symmetry is a duality between degenerating families of Calabi-Yau manifolds.
An algebraic family of complex varieties over a punctured disc gives rise naturally to a non-archimedean analytic space $X$ over the Laurent power series field $\mathbb C(\!(t)\!)$.
Given an SNC formal model $\fX$ of $X$ over the formal power series ring $\mathbb C[\![t]\!]$,
Vladimir Berkovich constructed a strong deformation retraction from $X$ onto a polytope $\SX$ inside $X$, which is homeomorphic to the dual intersection complex of the special fiber $\fXs$ (cf.\ \cite{Berkovich_Smooth_1999,Nicaise_Essential_skeleton_2013}).
In \cite{Kontsevich_Homological_2001} \S 3.3, Kontsevich and Soibelman proposed that Berkovich's deformation retraction can be regarded as an analog of the conjectural SYZ fibration in mirror symmetry.
We refer to \cite{Kontsevich_Affine_2006} and the works of Gross, Siebert, Hacking and Keel \cite{Gross_Real_Affine_2011,Gross_Tropical_2011,Gross_Mirror_Log_2011} for related developments along this direction.
Our considerations on non-archimedean Gromov compactness are to provide the foundations for the study of enumerative geometry in this setting.
In the subsequent paper \cite{Yu_Tropicalization_2014}, we study the tropicalization of the moduli stack of \kanal stable maps $\bcMgn(X,A)$ and its relation to the moduli space of tropical curves.
In the paper \cite{Yu_Enumeration_cylinders_2015}, we show how the general considerations here lead to new enumerative invariants for log Calabi-Yau surfaces.

\medskip
\paragraph{\textbf{Outline of the paper}}

In Section \ref{sec:basic_settings_(Gromov)}, we describe the basic settings of global tropicalization.
Given a $k$-analytic space $X$ and an SNC formal model $\fX$ of $X$,
we recall the construction of a continuous, proper, surjective map $\tau$ from $X$ to a simplicial complex $\SX$ (cf.\ \cite{Boucksom_Singular_2011,Yu_Balancing_2013,Gubler_Skeletons_2014,Kontsevich_Non-archimedean_2002}).
We call $\SX$ the \emph{Clemens polytope} and regard it as the tropicalization of $X$ with respect to the formal model $\fX$.
We prove the functoriality of tropicalization in Proposition \ref{prop:functoriality_of_Clemens}.

In order to make sure that the moduli space of curves in $X$ is of finite type, we need to impose certain positivity conditions on $X$.
A non-archimedean analogue of Kähler geometry was proposed by Kontsevich and Tschinkel in \cite{Kontsevich_Non-archimedean_2002}.
Our approach uses the settings in \cref{sec:basic_settings_(Gromov)} and follows mainly the suggestions in \cite{Kontsevich_Non-archimedean_2002}.
Basic notions of virtual line bundle, metrization, curvature and Kähler structure are defined in Section \ref{sec:virtual_line_bdles}.
In Section \ref{sec:curvature}, we prove the functoriality of curvature for metrized virtual line bundles.
In Section \ref{sec:virtual_line_bundles_on_a_curve}, we define the degree of virtual line bundles on curves.

Due to non-trivial automorphisms of objects, moduli spaces often carry the structure of stacks.
It also happens in our considerations.
In \cref{sec:definition_of_stacks}, we introduce the notions of formal stack and $k$-analytic stack.
We will use formal stacks as a tool for the construction of $k$-analytic stacks.
We prove that in the Deligne-Mumford case, the generic fiber of a proper formal stack is a proper $k$-analytic stack (\cref{thm:properness_of_analytic_stacks}).

In Section \ref{sec:artin_criteria}, we use Artin's representability criterion to construct the moduli stack of algebraic stable maps in full generality without the projectivity assumption (Theorem \ref{thm:algebraicity_of_stack_of_stable_maps}).

In Section \ref{sec:stacks_of_k-analytic_stable_maps},
we define formal stable maps and $k$-analytic stable maps.
We construct formal models for $k$-analytic stable maps in \cref{thm:formal_model_stable_maps}.
Then we construct the moduli stack of $k$-analytic stable maps by exhibiting it as the generic fiber of the moduli stack of formal stable maps (Theorem \ref{thm:moduli_of_analytic_stable_maps}).
The proof uses formal GAGA to reduce to the algebraic situation, followed by a careful study of the geometry of stable curves inspired by the work of de Jong on alterations \cite{de_Jong_Smoothness_1996,de_Jong_Families_1997,Abramovich_Alterations_2000}.

In Section \ref{sec:non-archimedean_Gromov_compactness}, we reduce the problem of non-archimedean Gromov compactness to the properness of the moduli stack of algebraic stable maps into the special fiber $\fXs$.
Then we prove the latter by constructing a proper covering using known compactness results in algebraic geometry.

\medskip
\paragraph{\bfseries Related works}

We refer to the works by Conrad and Temkin \cite{Conrad_Descent_for_coherent_2003,Conrad_Relative_2006,Conrad_Descent_2010} for the descent theory in non-archimedean geometry.

We would like to mention the related work by Abramovich, Caporaso and Payne \cite{Abramovich_Tropicalization_2012}.
They studied very interesting geometry of the moduli stack of stable curves from the non-archimedean analytic viewpoint.
We note that the moduli stack of stable curves is a special case of the moduli stack of stable maps where the target space is a point.
Other special cases of the moduli stack of stable maps are recently studied by Cavalieri, Markwig and Ranganathan \cite{Cavalieri_Tropicalizing_the_space_2014,Cavalieri_Tropical_compactification_2014,Ranganathan_Moduli_of_rational_2015}.

Besides the work by Kontsevich and Tschinkel \cite{Kontsevich_Non-archimedean_2002}, non-archimedean analogs of Kähler geometry were considered by Boucksom, Chambert-Loir, Favre, Gubler, Jonsson and Zhang \cite{Kontsevich_Non-archimedean_2002,Zhang_Positive_line_bundles_1995,Gubler_Local_heights_1998,Chambert-Loir_Mesures_2006,Chambert-Loir_Heights_2011,Boucksom_Singular_2011} with different viewpoints.

In \cite{Abramovich_Compactifying_2002}, Abramovich and Vistoli studied the moduli stack of twisted stable maps into a tame Deligne-Mumford stack which admits a projective coarse moduli scheme.

We are also inspired by the work of Hartl and Lütkebohmert \cite{Hartl_On_rigid-analytic_Picard_2000}, where they constructed rigid analytic Picard varieties.

\medskip

\paragraph{\bfseries Acknowledgments}
I am very grateful to Maxim Kontsevich for inspirations and guidance.
Special thanks to Antoine Chambert-Loir who provided me much advice and support.
I appreciate valuable discussions with Ahmed Abbes,  Denis Auroux, Pierrick Bousseau, Olivier Debarre, Antoine Ducros, Lie Fu, Ilia Itenberg, François Loeser, Johannes Nicaise, Mauro Porta, Matthieu Romagny, Michael Temkin and Jean-Yves Welschinger.
I would like to thank the referees for helpful comments.

\section{Basic settings of global tropicalization}\label{sec:basic_settings_(Gromov)}

In this section, we describe the basic settings of global tropicalization (cf.\ \cite{Boucksom_Singular_2011,Yu_Balancing_2013,Gubler_Skeletons_2014,Kontsevich_Non-archimedean_2002}).

Let $k$ denote a complete discrete valuation field, $k^\circ$ the ring of integers, $k^{\circ\circ}$ the maximal ideal of $k^\circ$, and $\tilde k$ the residue field.

We use the theory of non-archimedean analytic geometry in the sense of Vladimir Berkovich \cite{Berkovich_Spectral_1990,Berkovich_Etale_1993}.
All the $k$-analytic spaces we consider in this paper are assumed to be strictly $k$-analytic in the sense of \cite{Berkovich_Etale_1993}.

For an adic ring $\mathcal A$, we denote by $\Spf(\mathcal A)$ the formal spectrum of $\mathcal A$.
For a $k$-affinoid algebra $A$, we denote by $\SpB(A)$ the Berkovich spectrum of $A$.

For $0\leq d\leq n$ , $\mathbf m=(m_0,\dots,m_d)\in\Z^{d+1}_{>0}$ and $a\in k^{\circ\circ}\setminus 0$, put
\begin{equation}\label{eq:standard_formal_scheme}
\fS(n,d,\mathbf m,a) = \Spf \left(k^\circ\{T_0,\dots,T_{n}, T\inv_{d+1},\dots,T\inv_n\}/(T_0^{m_0}\cdots T_{d}^{m_d}-a)\right) .
\end{equation}
When $\mathbf m = (1,\dots,1)$, we denote $\fS(n,d,\mathbf m,a)$ simply by $\fS(n,d,a)$.

\begin{defin}[\cite{Berkovich_Smooth_1999}]
A formal scheme $\fX$ is said to be \emph{(locally) finitely presented} over $k^\circ$ if it is a (locally) finite union of open affine subschemes of the form \[\Spf\big(k^\circ\{T_0,\dots,T_n\}/(f_1,\dots,f_m)\big) .\]
\end{defin}

Let $\fX$ be a formal scheme locally finitely presented over $k^\circ$.
Its generic fiber $\fXe$ is a paracompact strictly $k$-analytic space; its special fiber $\fXs$ is a scheme locally of finite type over the residue field $\tilde k$ (cf.\ \cite{Berkovich_Vanishing_1994}).
We denote by $\pi\colon \fXe\rightarrow\fXs$ the reduction map from the generic fiber to the special fiber of $\fX$ (cf.\ \cite{Berkovich_Vanishing_1994} \S 1).

\begin{defin}
Let $X$ be a $k$-analytic space. A \emph{formal model} $\fX$ of $X$ consists of a formal scheme $\fX$ locally finitely presented over $k^\circ$ and an isomorphism between $X$ and the generic fiber $\fXe$ of $\fX$.
\end{defin}

\begin{defin}\label{def:gsss_formal_scheme}
A formal scheme $\fX$ finitely presented over $\kc$ is said to have \emph{simple normal crossings} (\emph{SNC} for short) if it satisfies the following conditions:
\begin{enumerate}[(i)]
\item Every point of $\fX$ has an open affine neighborhood $\fU$ such that the structure morphism $\fU\rightarrow\Spf k^\circ$ factors through an étale morphism $\phi\colon \fU\rightarrow\mathfrak S(n,d,\mathbf m,\varpi)$ for some $0\leq d\leq n$, $\mathbf m=(m_0,\dots,m_d)\in\Z^{d+1}_{>0}$ and a uniformizer $\varpi$ of $k$.
We require that $m_i$ does not equal to the characteristic of the residue field $\tilde k$ for every $i$.
\item All the intersections of the irreducible components of the special fiber $\fXs$ are either empty or geometrically irreducible.
\end{enumerate}
\end{defin}

Let $X$ be a $k$-analytic space and let $\fX$ be an SNC formal model\footnote{If the residue field $\tilde k$ has characteristic zero and if $X$ is compact quasi-smooth and strictly \kanal, then SNC formal models exist (cf.\ Temkin \cite{Temkin_Desingularization_2008}).} of $X$.
Let $\Set{D_i | i\in\IX}$ denote the set of the irreducible components of the special fiber $\fXs\red$ with its reduced scheme structure.
For every non-empty subset $I\subset\IX$, we put $D_I=\bigcap_{i\in I}D_i$.
We denote by $\mathit{mult}_i$ the multiplicity of $D_i$ in $\fXs$.

\begin{defin}
The \emph{Clemens polytope} $\SX$ is the simplicial subcomplex of the simplex $\Delta^{\IX}$ such that for every non-empty subset $I\subset\IX$, the simplex $\Delta^I$ is a face of $\SX$ if and only if the stratum $D_I$ is non-empty.
\end{defin}

One can construct a canonical inclusion map $\theta\colon\SX\hookrightarrow\fXe$ and a canonical strong deformation retraction from $\fXe$ to $\SX$ (cf.\ \cite{Berkovich_Smooth_1999,Nicaise_Essential_skeleton_2013}).
For the purpose of this paper, it suffices to know the construction of the retraction map $\tau\colon\fXe\to\SX$, i.e.\ the final moment of the strong deformation retraction.


\begin{defin}
A \emph{simple function} $\varphi$ on the Clemens polytope $\SX$ is a real valued function that is affine on every simplicial face of $\SX$.
For $i\in\IX$, we denote by $\varphi(i)$ the value of $\varphi$ at the vertex $i$.
\end{defin}

\begin{defin}
A \emph{vertical divisor} on $\fX$ is a Cartier divisor on $\fX$ supported on the special fiber $\fXs$.
\end{defin}

Let $\Div_0(\fX)_\R$ denote the vector space of vertical $\R$-divisors on $\fX$.
It is of dimension the cardinality of $\IX$.
An effective vertical divisor $D$ on $\fX$ is locally given by a function $u$ up to multiplication by invertible functions.
So $\val(u(x))$ defines a continuous function on $\fXe$ which we denote by $\varphi^0_D$.
The map $D\mapsto \varphi^0_D$ extends by linearity to a map from $\Div_0(\fX)_\R$ to the space of continuous functions $C^0(\fXe)$ on $\fXe$.

\begin{prop}\label{prop:map_to_Clemens}
Let $\tau\colon \fXe\rightarrow \Div_0(\fX)_\R^*$ be the evaluation map defined by \[\langle\tau(x),D\rangle=\varphi^0_D (x),\]
for any $x\in\fXe$, $D\in\Div_0(\fX)_\R$. Then
\begin{enumerate}[(i)]
\item The image of $\tau$ is naturally identified with the Clemens polytope $\SX$.
\item For any vertical $\R$-divisor $D=\sum a_i D_i\in \Div_0(\fX)_\R$, there exists a unique simple function $\varphi_D$ on $\SX$ such that $\varphi^0_D=\varphi_D\circ\tau$.
\end{enumerate}
\end{prop}
\begin{proof}
Let $\fU$ be an open affine subscheme of $\fX$ equipped with an étale morphism $\phi\colon \fU\rightarrow\fS(n,d,\mathbf m,\varpi)$ as in Definition \ref{def:gsss_formal_scheme}.
Let $D$ be a vertical divisor whose restriction to this chart is given by a function $f=T_0^{l_0}\cdots T_d^{l_d}$. We have
\[\varphi^0_D(x)=l_0\val(T_0(\phi(x)))+\dots+l_d\val(T_d(\phi(x)))\]
on the affinoid domain $\fU_\eta\subset\fXe$.
Therefore, the map $\tau$ restricted to $\fU_\eta$ is given by
\begin{equation}\label{eq:tau_local}
\tau\colon x\mapsto\left(\val T_0(x),\dots,\val T_d(x)\right) .
\end{equation}
The image of $\tau$ restricted to $\fU_\eta$ is the simplex given by the equation $\sum m_i x_i = 1$ in $\R^{d+1}_{\ge 0}$.
Now the statements (i) and (ii) are easily verified on this chart.
\end{proof}

\begin{rem}\label{rem:embedding}
Proposition \ref{prop:map_to_Clemens}(i) gives a canonical embedding 
\[\SX \subset \Div_0(\fX)_\R^*\simeq\R^{\IX}.\]
From now on, we always regard the Clemens polytope $\SX$ as being embedded in $\Div_0 (\fX)_\R^*$.
\end{rem}

We regard the Clemens polytope $\SX$ as the tropicalization of the $k$-analytic space $X$ with respect to the formal model $\fX$.
The following proposition shows the functoriality of tropicalization. (Compare \cite{Kontsevich_Non-archimedean_2002,Baker_Nonarchimedean_2011,Gubler_Skeletons_2014}.)

\begin{prop}\label{prop:functoriality_of_Clemens}
Let $\fX$, $\fY$ be two SNC formal schemes over $k^\circ$ and let $\SX$, $S_\fY$ be the corresponding Clemens polytopes.
Let $\tau_\fX\colon \fXe\rightarrow \SX$ and $\tau_\fY\colon \fYe\rightarrow S_\fY$ be the maps defined in Proposition \ref{prop:map_to_Clemens}.
Let $\ff\colon \fX\rightarrow\fY$ be a morphism of formal schemes. There exists a continuous map $S_\ff\colon \SX\rightarrow S_\fY$ such that the diagram
\begin{equation}
\begin{tikzcd}\label{eq:functoriality_of_Clemens}
\fXe \arrow{r}{\ff_\eta} \arrow{d}{\tau_\fX} & \fYe \arrow{d}{\tau_\fY} \\
\SX \arrow{r}{S_{\ff}} & S_\fY
\end{tikzcd}
\end{equation}
commutes. Moreover, the map $S_{\ff}$ is affine on every simplicial face of $\SX$.
\end{prop}
\begin{proof}
The pullback map $\ff^*\colon \Div_0(\fY)_\R\rightarrow\Div_0(\fX)_\R$ induces a linear map \[S_\ff\colon \Div_0(\fX)_\R^*\rightarrow\Div_0(\fY)_\R^*\] by duality. By the canonical embeddings $\SX\subset\Div_0(\fX)^*_\R$ and $S_\fY\subset\Div_0(\fY)^*_\R$ (cf.\ \cref{rem:embedding}), it suffices to show the commutativity of the following diagram
\begin{equation*}
\begin{tikzcd}
\fXe \arrow{r}{\ff_\eta} \arrow{d}{\tau_\fX} & \fYe \arrow{d}{\tau_\fY} \\
\Div_0(\fX)^*_\R \arrow{r}{S_{\ff}} & \Div_0(\fY)^*_\R .
\end{tikzcd}
\end{equation*}
In other words, it suffices to show that
\begin{equation}\label{eq:functoriality_of_Clemens_proof}
\varphi^0_{\ff^*(D)}(x)=\varphi^0_D(\ff_\eta(x))
\end{equation}
for any point $x\in\fXe$ and any effective vertical divisor $D\in \Div_0(\fY)$. Let $\pi_\fY\colon \fYe\rightarrow\fYs$ denote the reduction map. Assume that the divisor $D$ is cut out by some function $u$ near the point of reduction $\pi_\fY(\ff_\eta(x))$. Then the pullback $\ff^*(D)$ is cut out by $\ff^*u = u\circ \ff$. 
So \cref{eq:functoriality_of_Clemens_proof} is equivalent to the obvious identity
\[\val\big((u\circ \ff)(x)\big)=\val\big(u(\ff_\eta(x))\big) .\]
\end{proof}

\section{Metrizations and Kähler structures} \label{sec:virtual_line_bdles}

In order to make sure that the moduli space of stable maps is of finite type, we need to impose certain positivity conditions on the target space.
The definitions in this section follow mainly the suggestions by Kontsevich and Tschinkel in \cite{Kontsevich_Non-archimedean_2002} (see also \cite{Zhang_Positive_line_bundles_1995,Gubler_Local_heights_1998,Chambert-Loir_Mesures_2006,Chambert-Loir_Heights_2011,Boucksom_Singular_2011}).
The idea is to define structures on a \kanal space via structures on a Clemens polytope.

Let $X$ be a $k$-analytic space and $\fX$ an SNC formal model of $X$.
Let $\mathit{Simp}_\fX^\text{pre}$ be the presheaf on the Clemens polytope $\SX$, such that for every open subset $U$ of $\SX$, $\Simp^\mathrm{pre}_\fX(U)$ is the set of the restrictions to $U$ of all simple functions on $\SX$.
Let $\mathit{Simp}_\fX$ denote the associated sheaf, with respect to the usual topology on $\SX$.

\begin{defin}
For a smooth variety $V$ defined over the residue field $\tilde k$, let $\Div(V)$ denote the abelian group of divisors in $V$, and let $\Div^0(V)$ be the subgroup consisting of divisors whose intersection number with any proper curve in $V$ is zero.
Put $N^1(V)=\Div(V)/\Div^0(V)$ and $N^1(D_I)_\R=N^1(D_I)\otimes\R$.
An element $D\in N^1(V)$ is said to be \emph{nef} if its intersection number with any proper curve in $V$ is non-negative.
It is said to be \emph{ample} if it has a representative which is an ample divisor.
\end{defin}

\begin{defin}
Let $\Delta^I, I\subset\IX$ be a face of the Clemens polytope $\SX$.
For any simple function $\varphi$ on $\SX$, we define the \emph{derivative} of $\varphi$ with respect to $\Delta^I$
\[\partial_I \varphi = \sum_{i\in\IX}\mult_i\cdot\varphi(i)\cdot [D_i]|_{D_I} \in N^1(D_I)_\R ,\]
where $[D_i]$ denotes the divisor class of $D_i$.
The function $\varphi$ is said to be \emph{linear} (resp.\ \emph{convex}, \emph{strictly convex}) along the open simplicial face\footnote{An open simplicial face is by definition the relative interior of a face of the simplicial complex.} $(\Delta^I)^\circ$ corresponding to $I$ if $\partial_I \varphi$ is trivial (resp.\ nef, ample) in $N^1(D_I)_\R$.
Let $\mathit{LinLoc}(\varphi)$ (resp.\ $\mathit{ConvLoc}(\varphi)$, $\mathit{SConvLoc}(\varphi)$) be the union of the open simplicial faces in $\SX$ along which $\varphi$ is linear (resp.\ convex, strictly convex). For any subset $U\subset \SX$, the restriction $\varphi|_{U}$ is said to be \emph{linear} (resp.\ \emph{convex}, \emph{strictly convex}) if $\mathit{LinLoc}(\varphi)$ (resp.\ $\mathit{ConvLoc}(\varphi)$, $\mathit{SConvLoc}(\varphi)$) contains $U$.
\end{defin}

We denote by $\mathit{Lin}_\fX$ (resp.\ $\mathit{Conv}_\fX$, $\mathit{SConv}_\fX$) the subsheaf of $\mathit{Simp}_\fX$ whose germs are germs of linear (resp.\ convex, strictly convex) functions. The sheaf $\mathit{Lin}_\fX$ acts on the sheaf $\mathit{Simp}_\fX$ (resp.\ $\mathit{Conv}_\fX$, $\mathit{SConv}_\fX$) via additions
\[ \psi \longmapsto (\varphi\mapsto \varphi+\psi),\]
where $\psi$ is a local section of $\mathit{Lin}_\fX$ and $\varphi$ is a local section of $\mathit{Simp}_\fX$ (resp.\ $\mathit{Conv}_\fX$, $\mathit{SConv}_\fX$).

\begin{defin}
A \emph{virtual line bundle} $L$ on the $k$-analytic space $X$ with respect to the formal model $\fX$ is a torsor over the sheaf $\mathit{Lin}_\fX$. A \emph{simple} (resp.\ \emph{convex}, \emph{strictly convex}) \emph{metrization} $\hL$ of a virtual line bundle $L$ is a global section of the sheaf $\mathit{Simp}_\fX\otimes L$ (resp.\ $\mathit{Conv}_\fX\otimes L$, $\mathit{SConv}_\fX\otimes L$), where the tensor product is taken over the sheaf $\mathit{Lin}_\fX$.
\end{defin}

\begin{defin}\label{def:kahler}
A \emph{Kähler structure} $\hL$ on $X$ with respect to the formal model $\fX$ is a virtual line bundle $L$ on $X$ with respect to $\fX$ equipped with a strictly convex metrization $\widehat L$.
\end{defin}

In concrete terms, a simple metrization $\hL$ gives for each $i\in\IX$, a germ of a simple function $\varphi_i$ at the vertex $i$ of $\SX$ up to addition by linear functions. So we obtain a collection of numerical classes $\partial_i\varphi_i\in N^1(D_i)_\R$ for every $i\in\IX$.

\begin{defin}\label{def:curvature}
We call this collection of numerical classes $\partial_i\varphi_i\in N^1(D_i)_\R$ for every $i\in\IX$ the \emph{curvature} of the metrized virtual line bundle $\hL$, which we denote by $c(\hL)$.
\end{defin}

\begin{lem}\label{lem:compatibility_of_curvature}
The curvature $c(\widehat L)$ has the following compatibility property: for any simplicial face $\Delta^I, I\subset \IX$ and any two vertices $i,j\in I$, we have
\[(\partial_i \varphi_i)|_{D_I}=(\partial_j \varphi_j)|_{D_I}.\]
\end{lem}
\begin{proof}
Since $L$ is a torsor over the sheaf $\mathit{Lin}_\fX$, the simple functions $\varphi_i$ and $\varphi_j$ differ by a simple function which is linear along the open simplicial face $(\Delta^I)^\circ$. Therefore, we obtain the claimed identity.
\end{proof}

\section{Functoriality of curvature}\label{sec:curvature}

Let $\ff\colon \fX\rightarrow\fY$ be a morphism of SNC formal schemes over $k^\circ$.
Let $\Set{D_i|i\in\IX}$ and $\Set{D_j|j\in\IY}$ denote respectively the set of the irreducible components of $\fXs\red$ and $\fYs\red$ with their reduced structures. Let $c(\widehat L)$ be the curvature of a metrized virtual line bundle $\widehat L$ on $\fYe$ with respect to the formal model $\fY$.
By definition, $c(\widehat L)$ is a collection of numerical classes $c'_j\in N^1(D_j)_\R$ for every $j\in\IY$.
The morphism $\ff\colon \fX\rightarrow \fY$ induces a morphism of $\tilde k$-schemes $\ff_s\colon \fXs\rightarrow \fYs$ and a morphism $\ff\red_s\colon \fXs\red\rightarrow\fYs\red$ between their reduced structures.
For each $i\in\IX$, assume that the image $\ff\red_s(D_i)$ is contained in some $D_j$ for $j\in\IY$. Let $c_i = (\ff\red_s)^* c'_j \in N^1(D_i)_\R$.
By Lemma \ref{lem:compatibility_of_curvature}, the numerical class $c_i$ does not depend on the choice of $j\in\IY$.
We define the \emph{pullback} $\ff^*_s c(\widehat L)$ of the curvature $c(\widehat L)$ to be the collection of the numerical classes $c_i \in N^1(D_i)_\R$ for all $i\in\IX$.

By Proposition \ref{prop:functoriality_of_Clemens}, the morphism $\ff$ induces a map $S_\ff\colon \SX\rightarrow S_\fY$ which is affine on every simplicial face of $\SX$.

\begin{lem}\label{lem:pullback_of_linear_functions}
The pullback of a simple (resp.\ linear) function on $S_\fY$ by $S_{\ff}$ is a simple (resp.\ linear) function on $\SX$.
\end{lem}
\begin{proof}
Let $\varphi$ be a simple function on $S_\fY$ and let $S_{\ff}^*\varphi$ be the pullback of $\varphi$ by $S_{\ff}$.
The fact that $S_{\ff}$ is affine on every simplicial face of $\SX$ implies that $S_{\ff}^*\varphi$ is a simple function on $\SX$.
For any $I\subset\IX$, let $D_I=\bigcap_{i\in I}D_i$ denote the corresponding stratum of $\fXs\red$ with its reduced structure. Suppose that the image $\ff\red_s(D_I)$ is contained in the stratum $D_J$ of $\fY\red_s$ for some $J\subset\IY$. Let $(\ff\red _s|_{D_I})^*(\partial_J\varphi)$ denote the pullback of $\partial_J\varphi$ by $\ff\red _s|_{D_I}\colon D_I\rightarrow D_J$.
It suffices to prove the following lemma.
\end{proof}

\begin{lem}\label{lem:functoriality_of_derivative}
We have
\[(\ff\red _s|_{D_I})^*(\partial_J\varphi) = \partial_I (S_{\ff}^* \varphi) .\]
\end{lem}
\begin{proof}
By linearity and by the definition of the derivatives $\partial_I$, $\partial_J$, it suffices to show that for any effective vertical divisor $D\in\Div_0(\fY)$, we have
\[ (\ff\red _s|_{D_I})^* \big([D]|_{D_J}\big) = (\ff^*[D])|_{D_I} .\]
The identity above follows from the commutative diagram
\begin{equation*}
\begin{tikzcd}
D_I \arrow{r}{\ff\red_s} \arrow[hook]{d}& D_J \arrow[hook]{d} \\
\fX \arrow{r}{\ff} & \fY .
\end{tikzcd}
\end{equation*}
\end{proof}

Lemma \ref{lem:pullback_of_linear_functions} ensures that the notions of pullback of virtual lines bundles and pullback of their metrizations are well-defined.
Moreover, Lemma \ref{lem:functoriality_of_derivative} implies the following functoriality of curvature for metrized virtual line bundles.

\begin{prop}\label{prop:functoriality_of_curvature}
Let $\ff\colon \fX\rightarrow\fY$ be a morphism of SNC formal schemes over $k^\circ$, $L$ a virtual line bundle on $\fYe$ with respect to the formal model $\fY$, and $\hL$ a metrization of $L$.
The pullback of the curvature of $\widehat L$ equals the curvature of the pullback of $L$, i.e.
\[\ff^*_s c(\hL) = c(\ff^* (\widehat L)) .\]
\end{prop}

\section{Degree of virtual line bundles on curves}\label{sec:virtual_line_bundles_on_a_curve}

Let $X$ be a connected proper smooth $k$-analytic curve and let $\fX$ be an SNC formal model of $X$.
In this case, the Clemens polytope $\SX$ is a finite connected simple graph.

We fix an order on the set of vertices $I_\fX$. For two vertices $i$ and $j$, we write $i\mbox{---}j$ if $i$ and $j$ are connected by an edge. We write $i \prec j$ if $i$ and $j$ are connected by an edge and if $i$ is inferior to $j$ with respect to the fixed order. For every vertex $i$ of $\SX$, let $U_i$ be the union of the vertex $i$ and all the open edges whose closures contain $i$. It is an open neighborhood of the vertex $i$. If $e_{ij}$ is an edge with two endpoints $i$ and $j$, then $U_i\cap U_j$ is the interior of the edge $e_{ij}$.

\begin{lem}\label{lem:acyclicity}
We have
\begin{align}
\label{eq:cohomology1} &H^0\big(\mathit{Lin}_{\fX}|_{U_i}\big)\simeq\R^{d(i)},\quad H^q\big(\mathit{Lin}_{\fX}|_{U_i}\big)=0\\
\label{eq:cohomology2} &H^0\big(\mathit{Lin}_{\fX}|_{U_j\cap U_k}\big)\simeq\R\oplus\R, \quad H^q\big(\mathit{Lin}_{\fX}|_{U_j\cap U_k}\big)=0
\end{align}
for any $q\geq 1$, any vertex $i$, and any pair of vertices $j$, $k$ such that $j\prec k$. Here $d(i)$ denotes the number of edges connected to the vertex $i$.
\end{lem}
\begin{proof}
For a pair of vertices $j\prec k$, the stratum $D_{j,k}=D_j\cap D_k$ is a finite set of points.
Since a point does not have any non-trivial divisors on it, the sheaf $\mathit{Lin}_{\fX}|_{U_j\cap U_k}$ is equal to the sheaf $\mathit{Simp}_{\fX}|_{U_j\cap U_k}$. The latter sheaf is a constant sheaf of 2-dimensional real vector spaces on $U_j\cap U_k$. So we have the isomorphisms in \cref{eq:cohomology2}.

The sheaf $\mathit{Lin}_{\fX}|_{U_i}$ is more complicated. For every vertex $j$ which is connected to $i$ by an edge, we denote by $\iota_j$ the inclusion map $U_i\cap U_j\hookrightarrow U_i$. We have the adjunction morphism
\[\mathit{Lin}_{\fX}|_{U_i} \longrightarrow \iota_{j*}\iota^*_j \big(\mathit{Lin}_{\fX}|_{U_i}\big) \simeq \iota_{j*}\big(\R^2_{U_i\cap U_j}\big).\]
Taking direct sums over all vertices $j$ that are connected to $i$, we obtain a morphism
\[\mathit{Lin}_{\fX}|_{U_i} \longrightarrow\bigoplus_{j\mbox{---}i} \iota_{j*}\big(\R^2_{U_i\cap U_j}\big),\]
whose cokernel is a skyscraper  sheaf of rank $d(i)$ concentrated at vertex $i$.
Let us denote this skyscraper sheaf by $\R^{d(i)}_i$.
We obtain a short exact sequence of constructible sheaves on $U_i$
\[0\longrightarrow \mathit{Lin}_{\fX}|_{U_i} \longrightarrow \bigoplus_{j\mbox{---}i} \iota_{j*}\big(\R^2_{U_i\cap U_j}\big)\longrightarrow\R^{d(i)}_i\longrightarrow 0.\]
The cohomology groups are then easily calculated via the associated long exact sequence, and we obtain the isomorphisms in \eqref{eq:cohomology1}.
\end{proof}

We define a map
\[\mathit{deg}\colon \prod_{j\prec k}H^0\big(\mathit{Lin}_{\fX}|_{U_j\cap U_k}\big)\rightarrow\R\]
as follows.
An element of the left hand side is a collection of linear functions $\Set{\varphi_{jk} | j\prec k}$ on all the open edges.
We put
\begin{equation}\label{eq:def_of_deg}
\mathit{deg}\big(\{\varphi_{jk}\,|\,j\prec k\}\big) = \sum_{j\prec k} \mult_j\cdot\mult_k\cdot \big(\varphi_{jk}(k)-\varphi_{jk}(j)\big).
\end{equation}

\begin{prop}\label{prop:calc_H1}
The map $\mathit{deg}$ defined above induces an isomorphism
\[\deg\colon H^1\big(\mathit{Lin}_\fX \big)\xrightarrow{\ \sim\ }\R .\]
\end{prop}
\begin{proof}
By Lemma \ref{lem:acyclicity}, the cohomology of the sheaf $\mathit{Lin}_\fX$ can be calculated by the following ordered Čech complex with respect to the covering $\{U_i, i\in\IX\}$
\[0\rightarrow\prod_{i\in\IX} H^0\big(\mathit{Lin}_{\fX}|_{U_i}\big)\xrightarrow{d^0}\prod_{j\prec k} H^0\big(\mathit{Lin}_{\fX}|_{U_j\cap U_k}\big)\rightarrow 0.\]

Let $n_E$ denote the number of edges of $\SX$. By Lemma \ref{lem:acyclicity}, we have the isomorphisms
\[\prod_{i\in\IX} H^0\big(\mathit{Lin}_{\fX}|_{U_i}\big)\simeq \R^{2 n_E} ,\]
\[\prod_{j\prec k}H^0\big(\mathit{Lin}_{\fX}|_{U_j\cap U_k}\big)\simeq \R^{2 n_E}.\]

The differential $d^0$ is defined as follows. An element of $\prod_{i\in\IX} H^0\big(\mathit{Lin}_{\fX}|_{U_i}\big)$ is a collection of linear functions $\Set{\varphi_i| i\in\IX}$ on the open subsets $U_i$.
For any pair of vertices $j,k$ such that $j\prec k$, let $d^0\big(\{\varphi_i\,|\, i\in\IX\}\big)_{jk}$ denote the component of $d^0\big(\{\varphi_i\,|\, i\in\IX\}\big)$ in $H^0\big(\mathit{Lin}_{\fX}|_{U_j\cap U_k}\big)$.
We put
\[d^0\big(\{\varphi_i\,|\, i\in\IX\}\big)_{jk}=\varphi_j|_{U_j\cap U_k}-\varphi_k|_{U_j\cap U_k}.\]

Since the graph $\SX$ is connected, we deduce that $\Ker d^0$ consists of constant functions and $\dim_\R\Ker d^0=1$.
So we have $\dim_\R\Image d^0= 2 n_E-1$, and
\begin{equation}\label{eq:calc_H1_dim}
\dim_\R H^1 \big(\mathit{Lin}_\fX\big)=1.
\end{equation}

\begin{lem}\label{lem:calc_H1_proof}
We have $\mathit{deg}\circ d^0=0$.
\end{lem}
\begin{proof}
Let $\Set{\varphi_i|i\in\IX}$ be an element of $\prod_{i\in\IX} H^0\big(\mathit{Lin}_{\fX}|_{U_i}\big)$.
The linearity of the function $\varphi_i$ at the vertex $i$ implies that
\[\partial_i\varphi_i = \sum_{j\in\IX} \mult_j\cdot\varphi_i (j) \cdot [D_j]|_{D_i} = 0 \in N^1(D_i)_\R.\]
If the vertices $i$ and $j$ are not connected by an edge, $[D_j]|_{D_i} = 0$. So we have
\[\partial_i \varphi_i = \mult_i\cdot\varphi_i (i)\cdot [D_i]|_{D_i} + \sum_{j\mbox{---} i} \mult_j\cdot\varphi_i(j)\cdot[D_j]|_{D_i} = 0,\]
where $j \mbox{---} i$ means that the vertex $j$ and the vertex $i$ are connected by an edge. Since
\[0 = \sum_{j\in\IX} \mathit{mult}_j\cdot [D_j]|_{D_i} = \mathit{mult}_i\cdot [D_i]|_{D_i} + \sum_{j\mbox{---} i} \mathit{mult}_j\cdot [D_j]|_{D_i},\]
and
\[[D_j]|_{D_i} = 1\in N^1(D_i)_\R \quad\text{ for }j\mbox{---}i,\]
we obtain that
\begin{equation}\label{eq:linearity}
\varphi_i(i)\cdot\sum_{j\mbox{---}i} \mathit{mult}_j = \sum_{j\mbox{---} i}\mult_j\cdot\varphi_i(j).
\end{equation}
Therefore, we have
\begin{align*}
\MoveEqLeft (\mathit{deg}\circ d^0)\big(\{\varphi_i\,|\, i\in\IX\}\big)\\
&= \mathit{deg}\big(\{\varphi_j-\varphi_k\,|\, j\prec k\}\big)\\
&=\sum_{j\prec k}
\mult_j\cdot\mult_k\cdot\Big(\big(\varphi_j(k)-\varphi_k(k)\big)-\big(\varphi_j(j)-\varphi_k(j)\big)\Big)\\
&=\sum_{i\in\IX} \mult_i\cdot\bigg(-\varphi_i(i)\cdot\sum_{j\mbox{---}i} \mathit{mult}_j + \sum_{j\mbox{---}i}\mult_j\cdot\varphi_i(j)\bigg)=0
\end{align*}
\end{proof}

To conclude, Lemma \ref{lem:calc_H1_proof}, \cref{eq:calc_H1_dim} and the surjectivity of the map $\mathit{deg}$ imply that $\mathit{deg}$ induces an isomorphism
\[\deg\colon H^1\big(\mathit{Lin}_\fX\big)\xrightarrow{\ \sim\ }\R.\]
\end{proof}

Let $L$ be a virtual line bundle on the $k$-analytic curve $X$ with respect to the formal model $\fX$.
By definition, it corresponds to an element in $H^1(\mathit{Lin}_\fX)$.

\begin{defin}\label{def:degree_of_virtual_line_bundle}
The \emph{degree} of the virtual line bundle $L$ is defined to be the real number via the isomorphism in \cref{prop:calc_H1}. We denote it by $\deg(L)$.
\end{defin}

\begin{rem}
The degree does not depend on the choice of the ordering on the vertices of $\SX$.
One can check this using the alternating Čech complex.
\end{rem}

As in the case of complex geometry, the degree of a virtual line bundle can also be calculated via the curvature of its metrizations.
Let $\hL$ be a simple metrization of the virtual line bundle $L$. By Definition \ref{def:curvature}, the curvature $c(\hL)$ of $\hL$ is a collection of numerical classes $\partial_i \varphi_i\in N^1(D_i)_\R$ for every $i\in\IX$.

\begin{defin}\label{def:degree_of_curvature}
The \emph{degree} $\deg\big(c(\hL)\big)$ of the curvature $c(\hL)$ is the sum \[\sum_{i\in\IX} \mathit{mult}_i\cdot \deg \partial_i \varphi_i.\]
\end{defin}

\begin{prop}\label{prop:calc_of_degree}
For a virtual line bundle $L$ and a simple metrization $\widehat L$ of $L$, we have
\[\deg\big(c(\hL)\big) = \deg(L).\]
\end{prop}
\begin{proof}
Assume that the virtual line bundle $L$ is given by a collection of linear functions $\Set{\varphi_{jk} | j\prec k}$ on all the open edges, which we regard as transition functions.
By definition,
\begin{align*}
\deg(L) &= \mathit{deg}\big(\Set{\varphi_{jk} | j\prec k}\big) \\
&= \sum_{j\prec k}\mult_j\cdot\mult_k\cdot\big(\varphi_{jk}(k) - \varphi_{jk}(j)\big).
\end{align*}
On the other hand, assume that the simple metrization $\widehat L$ is given by a collection of simple functions $\Set{\varphi_i | i\in\IX}$ such that $\varphi_j-\varphi_k = \varphi_{jk}$ for any pair of vertices $j\prec k$.
As in the deduction of \cref{eq:linearity}, we have
\[\deg \partial_i\varphi_i = -\varphi_i(i) \sum_{j\mbox{---} i} \mathit{mult}_j + \sum_{j\mbox{---} i}\mult_j\cdot\varphi_i(j).\]
Therefore,
\begin{align*}
\deg\big(c(\hL)\big) &= \sum_{i\in\IX}\mathit{mult}_i\cdot\deg\partial_i\varphi_i\\
&= \sum_{i\in\IX} \mult_i\cdot\bigg(-\varphi_i(i) \sum_{j\mbox{---} i} \mathit{mult}_j + \sum_{j\mbox{---} i}\mult_j\cdot\varphi_i(j)\bigg)\\
&=\sum_{j\prec k}
\mult_j\cdot\mult_k\cdot\Big(\big(\varphi_j(k)-\varphi_k(k)\big)-\big(\varphi_j(j)-\varphi_k(j)\big)\Big)\\
&= \sum_{j\prec k} \mult_j\cdot\mult_k\cdot\big(\varphi_{jk}(k) - \varphi_{jk}(j)\big)\\
&= \deg L.
\end{align*}
\end{proof}

\section{Formal stacks and non-archimedean analytic stacks}\label{sec:definition_of_stacks}

Moduli spaces often carry the structure of stacks. 
In this section, we define the analogs of algebraic stacks in formal geometry and in non-archimedean analytic geometry.
We hope that our considerations serve as motivation for further studies on non-archimedean analytic stacks\footnote{After this paper, Mauro Porta and Tony Yue Yu developed a theory of analytic stacks in greater generality in \cite{Porta_Yu_Higher_analytic_stacks_2014}.
In the meantime, Martin Ulirsch also studied a notion of non-archimedean analytic stack in \cite{Ulirsch_Geometric_2014}.}.

Denote by $\FSch_{\kc}$ the category of formal schemes locally finitely presented over $k^\circ$, equipped with the étale topology.
Denote by $\An_k$ the category of strictly $k$-analytic spaces, equipped with the quasi-étale topology (cf.\ \cite{Berkovich_Vanishing_1994} \S 3).
Let $\mathrm{Set}$ denote the category of sets.

\begin{defin}
A \emph{formal stack} $\fX$ (locally) finitely presented over $k^\circ$ is a stack in groupoids over the site $\FSch_{\kc}$, such that the diagonal $\fX\rightarrow \fX\times_{\kc} \fX$ is representable\footnote{In general, one can allow the diagonal to be representable only by formal algebraic spaces.
We make the restrictive definition here because it suffices for our purposes and it simplifies the exposition.
\personal{We recall that any separated and locally quasi-finite morphisms of algebraic spaces are representable by schemes (\cite{Laumon_Champs_2000} Theorem A.2).
}} and that there exists a formal scheme $\fU$ (locally) finitely presented over $k^\circ$ and a smooth effective epimorphism $\fU\rightarrow \fX$.
\end{defin}

\begin{defin}\label{def:k-analytic_stack}
A \emph{strictly $k$-analytic stack} $X$ is a stack in groupoids over the site $\An_k$, such that the diagonal $X\rightarrow X\times_k X$ is representable and that there exists a strictly $k$-analytic space $U$ and a quasi-smooth\footnote{Antoine Ducros' work \cite{Ducros_Families_2011} is a comprehensive reference for the notions of flatness and smoothness in $k$-analytic geometry} effective epimorphism $U\rightarrow X$.
A strictly $k$-analytic stack $X$ is said to be \emph{compact} if one can choose the covering $U$ to be a compact strictly $k$-analytic space.
\end{defin}

As we assumed in the beginning of \cref{sec:basic_settings_(Gromov)} that all the $k$-analytic spaces we consider are strictly $k$-analytic, we will omit the word ``strictly'' from now on for simplicity.

\begin{rem}\label{rem:left_Kan_extension}
Let $(\cdot)\an$ denote the analytification functor from the category of schemes locally of finite type over $k$ to the category of $k$-analytic spaces,
let $(\cdot)_s$ denote the special fiber functor from the category of formal schemes locally finitely presented over $k^\circ$ to the category of schemes locally of finite type over the residue field $\tilde k$,
and let $(\cdot)_\eta$ denote the generic fiber functor from the category of formal schemes locally finitely presented over $k^\circ$ to the category of $k$-analytic spaces.
Using left Kan extension, the analytification functor $(\cdot)\an$ can be extended to a functor from the category of algebraic stacks locally of finite type over $k$ with representable diagonal to the category of $k$-analytic stacks, which we denote again by $(\cdot)\an$.
We refer to \cite[\S 6]{Porta_Yu_Higher_analytic_stacks_2014} for details.
Similarly, we obtain the special fiber functor $(\cdot)_s$ from the category of formal stacks locally finitely presented over $\Spf \kc$ to the category of algebraic stacks locally of finite type over $\tilde k$,
and we obtain the generic fiber functor $(\cdot)_\eta$ from the category of formal stacks locally finitely presented over $\Spf\kc$ to the category of $k$-analytic stacks.
\end{rem}

\begin{defin}
A morphism $\ff\colon \fX\rightarrow\fY$ of formal stacks locally finitely presented over $k^\circ$ is said to be \emph{proper} if the induced morphism $\ff_s\colon \fXs\rightarrow\fYs$ of algebraic stacks over $\tilde k$ is proper.
\end{defin}

\begin{defin}
A morphism $f\colon X\to Y$ of $k$-analytic stacks is said to be \emph{separated} if the diagonal morphism $\Delta_f\colon X\to X\times_Y X$ is representable by proper morphisms of \kanal spaces.
\end{defin}

\begin{defin}\label{def:properness_of_analytic_stacks}
A morphism $f\colon X\to Y$ of $k$-analytic stacks is said to be \emph{proper} if it is separated, and if there exists an affinoid quasi-smooth covering $\{Y_i\}_{i\in I}$ of $Y$ such that for every $i$ there exists two finite affinoid quasi-smooth coverings $\{U_{ij}\}_{j\in J_I}$ and $\{V_{ij}\}_{j\in J_I}$ of $X\times_Y Y_i$ such that $U_{ij}\subset\Int(V_{ij}/Y_i)$ for every $j$.
\end{defin}

\begin{rem}
\cref{def:properness_of_analytic_stacks} is inspired by the definition of properness in rigid analytic geometry (cf.\ \cite{Bosch_Non-archimedean_1984} \S 9.6.2).
\end{rem}

\begin{thm}\label{thm:properness_of_analytic_stacks}
Let $\ff\colon \fX\rightarrow\fY$ be a proper morphism of formal stacks locally finitely presented over $k^\circ$.
Assume that the diagonal $\Delta_{\ff}\colon \fX\to \fX\times_{\fY} \fX$ is unramified.
Then the induced morphism of the generic fibers $\ff_\eta\colon\fXe\to\fYe$ is a proper morphism of $k$-analytic stacks.
\end{thm}

\begin{lem}\label{lem:reduction_of_algebraic_stacks}
Let $X$ be a locally noetherian algebraic stack.
Assume that the diagonal $\Delta_X\colon X\to X\times X$ is locally of finite type and locally separated.
Let $X\red$ denote the reduction of $X$, i.e.\ the reduced algebraic stack structure on $X$.
If $X\red$ is a scheme, then the algebraic stack $X$ is also a scheme.
\end{lem}
\begin{proof}\footnote{I am very grateful to Matthieu Romagny for his help with the proof of the lemma.}
Let $\cI_X\coloneqq X\times_{X\times X} X$ denote the inertia stack, and let $p\colon \cI_X\to X$ be the projection.
The morphism $p$ is representable by group algebraic spaces.
By the following cartesian diagram
\begin{equation*}
\begin{tikzcd}
\cI_X \arrow{r}{\sim} \arrow{d}{p} & X\times_{X\times X} X \arrow{r}\arrow{d} & X\times X \arrow{d}\\
X \arrow{r}{\sim} & X\times_X X \arrow{r} & (X\times X)\times_X (X\times X),
\end{tikzcd}
\end{equation*}
the morphism $p$ is locally of finite type and locally separated.
Since $X\red$ is a scheme, the morphism $p\red\colon \cI_{X\red}\to X\red$ is an isomorphism.
In particular, $p\red$ is representable by affine group schemes.
By \cite{Knutson_Algebraic_spaces_1971} Chap.\ III Corollary 3.6, 
the morphism $p$ is representable by group schemes.
By \cite{EGA1} Chap.\ I Proposition 5.1.9, the morphism $p$ is moreover representable by affine group schemes.

Let $\cM\subset\mathcal O_{\cI_X}$ denote the ideal sheaf cutting out the image of the identity section $e\colon X\to \cI_X$.
The fact that $p\red$ is an isomorphism implies that $\cM\otimes_{\mathcal O_{X}} \mathcal O_{X\red}=0$.
Thus $\cM=0$ by Nakayama Lemma.
So $p\colon\cI_X\to X$ is an isomorphism.
Therefore, by \cite{Laumon_Champs_2000} Corollary 8.1.1, the algebraic stack $X$ is an algebraic space.
Now by \cite{Knutson_Algebraic_spaces_1971} Chap.\ III Corollary 3.6 again, the algebraic space $X$ is a scheme.
\end{proof}

\begin{proof}[Proof of \cref{thm:properness_of_analytic_stacks}]
The question being local on the base, we can assume that $\fY$ is affine without loss of generality.
In this case, $\fX$ is a formal stack finitely presented over $k^\circ$, and $\fXs$ is an algebraic stack of finite type over $\tilde k$.
By \cite{Laumon_Champs_2000} Corollary 16.6.1 or \cite{Olsson_Proper_2005}, there exists a quasi-projective scheme $\cX^0$ over $\fYs$ and a proper surjective morphism $\cX^0\rightarrow\fXs$ over $\fYs$.
Since by assumption, the induced morphism $\ff_s\colon \fXs\rightarrow\fYs$ is a proper morphism of algebraic stacks of finite type over $\tilde k$, the scheme $\cX^0$ is projective over $\fYs$.
Let $\iota_0\colon\cX^0\hookrightarrow\bP^N_{\fYs}$ be a closed embedding.

For any scheme $T$ and any morphism $T\to\fXs$, we have a cartesian diagram
\begin{equation*}
\begin{tikzcd}
\cX^0\times_{\fXs} T \arrow{r} \arrow{d}
& \fXs \arrow{d}{\Delta_{\ff_s}}\\
\cX^0\times_{\fYs} T \arrow{r}
& \fXs\times_{\fYs} \fXs.
\end{tikzcd}
\end{equation*}

Since $\Delta_{\ff}\colon\fX\to\fX\times_\fY\fX$ is unramified by assumption,
the morphism $\Delta_{\ff_s}\colon\fXs\to\fXs\times_{\fYs}\fXs$ and the morphism $\cX^0\times_{\fXs} T\to \cX^0\times_{\fYs} T$ are also unramified.
The closed embedding $\iota_0\colon\cX^0\hookrightarrow\bP^N_{\fYs}$ induces a closed embedding
\[\cX^0\times_{\fYs} T\hookrightarrow\bP^N_{\fYs}\times_{\fYs} T.\]
Therefore the composition
\[\cX^0\times_{\fXs} T \longrightarrow \cX^0\times_{\fYs} T \longrightarrow \bP^N_{\fYs} \times_{\fYs} T\simeq \bP^N_{\fYs} \times_{\fYs} \fXs \times_{\fXs} T\simeq \bP^N_{\fXs}\times_{\fXs} T\]
is unramified.
In other words, the morphism $\cX^0\to\bP^N_{\fXs}$ is unramified.

Since étale locally on the source and target, unramified morphisms are closed embeddings,
it makes sense to take formal completion of $\bP^N_\fX$ along $\cX^0$.
We denote the formal completion by $\fX^0$.
\cref{lem:reduction_of_algebraic_stacks} implies that the formal stack $\fX^0$ is in fact a formal scheme.
It is a special formal scheme in the sense of \cite{Berkovich_Vanishing_II_1996}.
Let $\iota_\eta\colon\fXe^0\to\bP^N_{\fXe}$ denote the induced morphism on the generic fiber.
Let $p\colon\fXe^0\to\fXe$ be the composition of $\iota_\eta$ with the projection $\bP^N_{\fXe}\to\fXe$.

Since the formal stack $\fX$ is finitely presented over $k^\circ$, we can choose a smooth covering of $\fX$ by an affine formal scheme $\fA$ finitely presented over $k^\circ$.
Let $A$ denote the generic fiber of $\fA$.
It is a compact $k$-analytic space.
Let 
\begin{align*}
\iota_A&\colon\fXe^0\times_{\fXe} A\to\bP^N_{\fXe}\times_{\fXe} A\simeq \bP^N_A,\\
p_A&\colon\fXe^0\times_{\fXe} A\to \fXe\times_{\fXe} A\simeq A
\end{align*}
be the base change of the morphisms $\iota_\eta$ and $p$.

By construction, quasi-étale locally on the target, the morphism $\iota_A$ is the composition of an open immersion and a quasi-étale morphism.
So it is quasi-étale by faithfully flat descent \cite{Conrad_Descent_2010}.
Since the projection $\bP^N_A\to A$ is smooth, and the morphism $p_A$ is the composition of $\iota_A$ with the projection $\bP^N_A\to A$, it is then quasi-smooth.
Since the covering $\cX^0\to\fXs$ is proper, i.e.\ the base change $\cX^0\times_{\fXs}\fA_s\to \fA_s$ is proper,
the morphism $p_A$ is boundaryless by \cite{Temkin_Local_2000}.
Moreover, as $A$ is affinoid, the $k$-analytic space $\fXe^0\times_{\fXe} A$ is good.
Therefore, the quasi-smoothness and the boundarylessness of the morphism $p_A$ imply that $p_A$ is smooth (\cite{Ducros_Families_2011} (4.4.4)).
In particular, the morphism $p_A$ is open.

Since the scheme $\cX^0$ is proper over $\fYs$, the composition of $\fXe^0\xrightarrow{p}\fXe\to\fYe$ is boundaryless by \cite{Temkin_Local_2000}.
Therefore, for every point $x\in\fXe^0$, there exists two affinoid neighborhoods $U_x, V_x$ of the point $x$ in $\fXe^0$ such that $U_x\subset\Int(V_x/\fYe)$.
By construction, the surjectivity of the covering $\cX^0\to\fXs$ implies the surjectivity of the morphism $p_A$.
For every point $j\in A$, let $x(j)$ be a point in $\fXe^0\times_{\fXe} A$ which maps to $j$.
Let $\bar{x}(j)$ be the image of $x(j)$ under the projection $\fXe^0\times_{\fXe} A\to\fXe^0$.
Then $U_{\bar{x}(j)}\times_{\fXe} A$ is a neighborhood of $x(j)$ in $\fXe^0\times_{\fXe} A$.
Let $U'_j$ denote the image of the projection $U_{\bar{x}(j)}\times_{\fXe} A\to A$.
It is a neighborhood of $j$ in $A$ by the openness of the morphism $p_A$.

By the compactness of the $k$-analytic space $A$, there exists a finite set of points $J\subset A$ such that $\{U'_j\}_{j\in J}$ is a finite covering of $A$.
Then $\{U_{\bar{x}(j)}\}_{j\in J}$ and $\{V_{\bar{x}(j)}\}_{j\in J}$ are two finite affinoid quasi-smooth coverings of the $k$-analytic stack $\fXe$ such that $U_j\subset\Int(V_j/\fYe)$.

Now we prove the separatedness of the morphism $\ff_\eta\colon \fXe\rightarrow\fYe$.
The properness of the diagonal $\Delta_{\ff_\eta}\colon\fXe\to\fXe\times_{\fYe}\fXe$ follows from the properness of the diagonal $\Delta_{\ff_s}\colon\fXs\to\fXs\times_{\fYs}\fXs$ (see \cite{Temkin_Local_2000,Lutkebohmert_Formal_1990}).
Combining the double covering $\{U_{\bar{x}(j)}\}_{j\in J}$, $\{V_{\bar{x}(j)}\}_{j\in J}$, and the separatedness, we have proved the properness of the morphism~$\ff_\eta$.
\end{proof}

\begin{conj}\label{conj:properness_of_analytic_stacks}
We conjecture that \cref{thm:properness_of_analytic_stacks} holds without the assumption on the unramifiedness of the diagonal.
\end{conj}

\section{The moduli stack of algebraic stable maps}\label{sec:artin_criteria}

The notion of stable map was introduced by Maxim Kontsevich \cite{Kontsevich_Enumeration_1995} in order to compactify the moduli space of smooth curves in an algebraic variety and to provide an algebro-geometric foundation for Gromov-Witten theory.
In this section, we construct the moduli stack of stable maps for a general scheme locally of finite presentation over a locally noetherian base.
This generality will be useful for the study of $k$-analytic stable maps in Section \ref{sec:stacks_of_k-analytic_stable_maps}.

Fix two non-negative integers $g$ and $n$. 
Fix a locally noetherian scheme $S$ and denote by $\Sch_S$ the category of schemes over $S$. Let $X$ be a scheme locally of finite presentation over $S$.

\begin{defin}[cf.\ \cite{Kontsevich_Enumeration_1995} \S 1.1, \cite{Abramovich_Stable_2001} \S 2]\label{def:stable_maps_(Gromov)}
Let $T$ be an $S$-scheme. An \emph{\gn (algebraic) stable map} $\big(C\rightarrow T, (s_i), f\big)$ into $X$ over $T$ consists of a morphism $C\rightarrow T$, a morphism $f\colon C\rightarrow X$ and $n$ morphisms $s_i\colon T\rightarrow C$ such that
\begin{enumerate}[(i)]
\item The morphism $C\rightarrow T$ is a proper flat family of curves;
\item The geometric fibers of $C\rightarrow T$ are reduced with at worst double points as singularities, and are of arithmetic genus $g$;
\item The $n$ morphisms $s_i\colon T\rightarrow C$ are disjoint sections of $C\rightarrow T$ which land in the smooth locus of $C\rightarrow T$;
\item (Stability condition) For any geometric fiber $C_t$ of $C\rightarrow T$, every irreducible component of $C_t$ of genus 0 (resp.\ 1) which maps to a point under $f$ must have at least 3 (resp.\ 1) special points on its normalization,
where special points mean either marked points or the points coming from the double points.
\end{enumerate}
\end{defin}

\begin{defin}
If we remove Condition (iv) of Definition \ref{def:stable_maps_(Gromov)}, we call the triple $\big(C\rightarrow T, (s_i), f\big)$ an \emph{\gn pre-stable map} into $X$ over $T$.
\end{defin}

\begin{defin}
A \emph{morphism of stable maps} $\big(C\rightarrow T,(s_i),f\big)\longrightarrow\big(C'\rightarrow T',(s'_i), f'\big)$ is a commutative diagram
\begin{equation*}
\begin{tikzcd}
C \arrow{r} \arrow{d}
& C' \arrow{d}\\
T \arrow{r}
& T'
\end{tikzcd}
\end{equation*}
inducing an isomorphism $C\xrightarrow{\sim} C'\times_T T'$, compatible with the sections $s_i, s'_i$ and the morphisms $f, f'$.
\end{defin}

\begin{defin}\label{def:stable_curves}
Let $T$ be an $S$-scheme. An \emph{\gn (algebraic) stable (resp.\ pre-stable) curve} $\big(C\rightarrow T, (s_i)\big)$ over $T$ is an \gn stable (resp.\ pre-stable) map $\big(C\rightarrow T, (s_i), f\big)$ into $S$ over $T$ with $f$ being the structure morphism $C\rightarrow S$.
\end{defin}

In the classical case where the base $S$ is a field and $X$ is projective over $S$, one fixes an element $\beta$ in the cone of curves $\NE(X)$.
Denote by $\bcMgn(X/S,\beta)$ the category, fibered in groupoids over the category $\Sch_S$, of \gn stable maps into $X$ with class $\beta$. The main property of this category is described by the following theorem.

\begin{thm}[\cite{Kontsevich_Enumeration_1995}]\label{thm:stack_of_stable_maps_projective_case}
When $S$ is a field and $X$ is projective over $S$, the category $\bcMgn(X/S,\beta)$ is a proper algebraic stack.
\end{thm}

More details can be found in \cite{Fulton_Stable_1997,Abramovich_Stable_2001,Behrend_Stacks_GW_1996}.
For our purpose, we need to work over a more general base scheme $S$ and weaken the projectivity assumption. The statement of Theorem \ref{thm:stack_of_stable_maps_projective_case} consists of two parts: algebraicity and properness. We will discuss the algebraicity in this section and the properness in \cref{sec:non-archimedean_Gromov_compactness}.

In the above-mentioned literatures, the algebraicity is shown by projective methods which rely on the assumption that $X$ is projective over $S$. More precisely, the projectivity of $X$ induces projective embeddings of the pre-stable curves at the source via very ample line bundles.
The moduli stack of stable maps is then realized as a quotient stack of a subscheme of a certain Hilbert scheme.
In order to prove the algebraicity without the projectivity assumption, we will use Artin's deformation-theoretic criterion of representability.

\begin{thm}\label{thm:algebraicity_of_stack_of_stable_maps}
Let $X$ be a scheme locally of finite presentation over a locally noetherian base scheme $S$.
The category $\bcMgn(X/S)$, fibered in groupoids over the category $\Sch_S$, of \gn stable maps into $X$ is an algebraic stack locally of finite presentation over $S$.
\end{thm}

The following remark enables us to describe stable maps in terms of coherent sheaves.

\begin{rem}\label{rem:stable_maps_as_sheaves}
Let $\big(C\rightarrow T,(s_i), f\big)$ be an \gn stable map into $X$ over $T$. Let $X_T=X\times_S T$, $f_T\colon C\rightarrow X_T$, and
\[\mathcal L = \bigoplus_{m\geq 0} f_{T*}\Big(\omega_{C/T}\big(\sum S_i\big)\Big)^{\otimes m},\]
where $\omega_{C/T}$ denotes the relative canonical sheaf and $S_i$ denotes the image of the section $s_i$.
Then $\mathcal L$ is a quasi-coherent sheaf of graded algebras over $X_T$ and every direct summand of $\mathcal L$ is coherent.
The stability condition implies that the invertible sheaf $\omega_{C/T}\big(\sum S_i\big)$ is relatively ample over $X_T$.
Therefore, via the relative Proj construction, we recover the total space of the family of curves $C\simeq \Proj_{X_T}(\mathcal L)$, as well as the morphisms $C\rightarrow T$ and $f\colon C\rightarrow X$.
Moreover, the sections $(s_i)$ are determined by their images $(S_i)$, whose structure sheaves are quotients of the structure sheaf of $C$.
To conclude, the data of an \gn stable map into $X$ over $T$ can be encoded entirely in terms of coherent sheaves.
Conversely, assume that we are given a quasi-coherent sheaf $\mathcal L$ of graded algebras over $X_T$ which is a direct sum of coherent sheaves, and that we are also given $n$ sections of the morphism $\Proj_{X_T}(\mathcal L)\rightarrow T$.
In order to check if they come from an \gn stable map into $X$ over $T$ by the constructions above, it suffices to verify the following conditions:
First, $\Proj_{X_T}(\mathcal L)$ is flat over $T$;
Second, every geometric fiber over $T$ is an \gn stable map into $X$ over an algebraically closed field.
We note that both conditions are fpqc local on $T$.
\end{rem}

We study the properties of the category $\bcMgn(X/S)$ in Lemmas \ref{lem:sheaf_property} - \ref{lem:deformation_obstruction}.

\begin{lem}[sheaf property]\label{lem:sheaf_property}
The category $\bcMgn(X/S)$ is a stack in groupoids over the category $\Sch_S$ for the fpqc topology.
\end{lem}
\begin{proof}
The Isom presheaf is clearly a sheaf on the category $\Sch_S$ for the fpqc topology.
Let us now check the effectiveness of descent. 
Let $\pi\colon T'\rightarrow T$ be an fpqc covering and let $\big(C'\rightarrow T',(s'_i),f'\big)$ be a stable map into $X$ over $T'$ equipped with descent data relative to $\pi$. Denote $X_T=X\times_S T, X_{T'}=X\times_S T'$.
By base change, we obtain an fpqc covering $\pi_X\colon X_{T'}\rightarrow X_T$ and a morphism $C'\rightarrow X_{T'}$ equipped with descent data relative to $\pi_X$. Denote by $S'_i$ the image of $s'_i$.
As in Remark \ref{rem:stable_maps_as_sheaves}, the stability condition implies that the invertible sheaf $\omega_{C'/T'}\big(\sum S'_i\big)$ is relatively ample over $X_{T'}$. Therefore, the morphism $C'\rightarrow X_{T'}$ descends to a morphism $C\rightarrow X_T$ by \cite{SGA1} Chapitre VIII Proposition 7.8. So the morphism $f'\colon C'\rightarrow X$ descends to a morphism $f\colon C\rightarrow X$ as well. The sections $s_i'\colon T'\rightarrow C'$ descend to sections $s_i\colon T\rightarrow C$ by Theorem 5.2 in loc.\ cit.
The conditions for $\big(C'\rightarrow T',(s'_i),f'\big)$ to be a stable map also descend effectively to the triple $\big(C\rightarrow T,(s_i), f\big)$ by Remark \ref{rem:stable_maps_as_sheaves}. So we have proved the lemma.
\end{proof}

\begin{rem}
One may try to descend the family $C'\rightarrow T'$ of pre-stable curves with respect to the fpqc covering $T'\rightarrow T$ in the beginning. However, this is not possible à priori in the category of schemes instead of algebraic spaces. The stability condition ensures the effectiveness of descent in Lemma \ref{lem:sheaf_property}.
\end{rem}

\begin{lem}[limit preserving]\label{lem:limit_preserving}
The stack $\bcMgn(X/S)$ over the category $\Sch_S$ is limit preserving, i.e.\ the canonical map
\[\varinjlim\bcMgn(X/S)(T_i) \longrightarrow \bcMgn(X/S)(\varprojlim T_i)\]
is an equivalence of groupoids for any inverse system of affine schemes $\{T_i\}$.
\end{lem}
\begin{proof}
This follows from the standard properties of schemes of finite presentation (cf.\ \cite{EGA4-3} Theorem 8.8.2).
\end{proof}

\begin{lem}[Rim-Schlessinger condition for small affine pushouts (cf.\ \cite{stacks-project} Tag 07WP)]\label{lem:Schlessinger}
Let $T,T_1,T_2,T_{12}$ be spectra of local Artinian rings of finite type over $S$. Assume that $T\rightarrow T_1$ is a closed immersion and that
\begin{equation*}
\begin{tikzcd}
T \arrow{r} \arrow{d}
& T_1 \arrow{d}\\
T_2 \arrow{r}
& T_{12}=T_1\coprod_T T_2
\end{tikzcd}
\end{equation*}
is a pushout diagram in the category $\Sch_S$.
Then the canonical map
\[\bcMgn(X/S)(T_{12})\xrightarrow{\ \sim\ }\bcMgn(X/S)(T_1)\times_{\bcMgn(X/S)(T)} \bcMgn(X/S)(T_2)\]
is an equivalence of groupoids.
\end{lem}
\begin{proof}
Using Remark \ref{rem:stable_maps_as_sheaves}, the pushout property for stable maps follows from the pushout property for quasi-coherent sheaves (cf.\ \cite{stacks-project} Tag 08LQ, Tag 08IW).
\end{proof}

\begin{lem}[effectiveness of formal objects]\label{lem:effectiveness}
Let $\widehat A$ be a complete local algebra over $S$, with maximal ideal $\mathfrak m$, and with residue field of finite type over $S$. Then the canonical map
\[\bcMgn(X/S)(\widehat A)\longrightarrow\varprojlim_l \bcMgn(X/S)(\widehat A/\mathfrak m^l)\]
is an equivalence of groupoids.
\end{lem}
\begin{proof}
Given a formal object $\Big\{\big(C^{(l)}\rightarrow\Spec\widehat A/\mathfrak m^l,(s_i^{(l)}),f^{(l)}\big)\Big\}$ on the right hand side, using Remark \ref{rem:stable_maps_as_sheaves}, by Grothendieck's existence theorem for formal sheaves, one can construct a triple $\big(\widehat C\rightarrow\Spec\widehat A,(\widehat s_i),\widehat f\big)$ that induces the formal object by restrictions.
It remains to see that the triple $\big(\widehat C\rightarrow\Spec\widehat A,(\widehat s_i),\widehat f\big)$ is a stable map into $X$ over $\Spec\widehat A$. This follows from the fact that the stability condition (Definition \ref{def:stable_maps_(Gromov)}(iv)) for a pre-stable map $\big(C\rightarrow T, (s_i), f\big)$ into $X$ over $T$ is an open condition on the base $T$.
So we have shown that the canonical map in the statement of the lemma is essentially surjective.
The full faithfulness is also implied by Grothendieck's existence theorem along the same reasoning.
\end{proof}

\begin{lem}[local quasi-separation]\label{lem:diagonal}
Let $\big(C\rightarrow T, (s_i), f\big)$ be an element of $\bcMgn(X/S)(T)$ and let $\phi$ be an automorphism of the element.
If $\phi$ induces the identity in $\bcMgn(X/S)(t)$ for a dense set of points $t\in T$ of finite type, then $\phi$ is the identity.
\end{lem}
\begin{proof}
If $\phi$ induces the identity in $\bcMgn(X/S)(t)$ for a dense set of points $t\in T$ of finite type, then $\phi$ is the identity on a dense set of points of finite type on $C$. Therefore, $\phi$ is the identity because $C$ is separated over $T$.
\end{proof}

The following lemma describes the deformation theory of the category $\bcMgn(X/S)$.

\begin{lem}[deformation and obstruction]\label{lem:deformation_obstruction}
Let $A$ be an $S$-algebra, $M$ an $A$-module, and $x=\big(C\rightarrow\Spec A,(s_i),f\big)$ an element of $\bcMgn(X/S)(\Spec A)$. Denote by $A[M]$ the $A$-algebra whose underlying $A$-module is $A\oplus M$ and whose multiplication is given by $(a,m)\cdot(a',m')=(aa',am'+a'm)$.
We consider the deformation situation $(A[M]\twoheadrightarrow A, x)$. Denote by $L$ the relative cotangent complex $L_{C/X\times_S\Spec A}$ and by $\mathcal N$ the normal sheaf of the union of the images of the sections $(s_i)$ in $C$. We have the following:
\begin{enumerate}[(i)]
\item The module of infinitesimal automorphisms
$\Aut_x(M)$ sits in the exact sequence
\[0\longrightarrow \Aut_x(M)\longrightarrow \Hom(L,\mathcal O_C\otimes_A M)\longrightarrow H^0(C, \mathcal N\otimes M).\]
\item The module of infinitesimal deformations $D_x(M)$ sits in the exact sequence
\[\Hom(L,\mathcal O_C\otimes_A M)\rightarrow H^0(C, \mathcal N\otimes M)\rightarrow D_x(M)\rightarrow \Ext^1(L,\mathcal O_C\otimes_A M)\rightarrow 0.\]
\item The module of obstructions $\mathcal O_x(M)$ can be given as $\Ext^2(L,\mathcal O_C\otimes_A M)$.
\end{enumerate}
\end{lem}
\begin{proof}
If we ignore the sections $(s_i)$ and denote $\overline x = (C/\Spec A,f)$, then the module of infinitesimal automorphisms $\Aut_{\overline x} (M)$, the module of infinitesimal deformations $D_{\overline x}(M)$ and the module of obstructions $\mathcal O_{\overline x}(M)$ are given by $\Hom(L,\mathcal O_C\otimes_A M)$, $\Ext^1(L,\mathcal O_C\otimes_A M)$ and $\Ext^2(L,\mathcal O_C\otimes_A M)$ respectively. Now we take the sections $(s_i)$ into consideration.
The module of infinitesimal automorphisms $\Aut_x(M)$ consists of those elements of $\Aut_{\overline x}(M)$ that fix the sections $(s_i)$, so the statement (i) in the lemma is true. The module of infinitesimal deformations $D_x(M)$ is the extension of the module $D_{\overline x}(M)$ by the module of infinitesimal deformations of the sections $H^0(C, \mathcal N\otimes M)$.
So the statement (ii) is true.
Finally, since the sections $(s_i)$ land in the smooth locus of $C\rightarrow T$, infinitesimal deformations of the sections are unobstructed.
So we have $\mathcal O_x(M) = \mathcal O_{\overline x}(M)$, i.e.\ the statement (iii) is true.
\end{proof}

\begin{proof}[Proof of \cref{thm:algebraicity_of_stack_of_stable_maps}]
We verify Artin's representability criterion for algebraic stacks as stated in \cite{Artin_Versal_1974} Theorem 5.3\footnote{The paper \cite{Artin_Versal_1974} has the assumption that the base scheme $S$ is of finite type over a field or an excellent Dedekind domain. It is generalized to the locally noetherian case in \cite{stacks-project} Tag 07SZ.}.
Lemmas \ref{lem:sheaf_property} and \ref{lem:limit_preserving} show that $\bcMgn(X/S)$ is a limit preserving stack over the category $\Sch_S$. Lemma \ref{lem:Schlessinger} verifies the first part of Condition (1) of Theorem 5.3 in loc.\ cit.\ (see \cite{Artin_Versal_1974} (2.3) Condition (S1')).
Lemma \ref{lem:effectiveness} verifies Condition (2).
Using the properties of cotangent complexes (\cite{Illusie_Complexe_1971}) and by the five lemma, Lemma \ref{lem:deformation_obstruction} verifies Condition (3) and the remaining parts of Condition (1).
Lemma \ref{lem:diagonal} verifies Condition (4).
To conclude, we have proved that the category $\bcMgn(X/S)$ is an algebraic stack locally of finite presentation over $S$.
\end{proof}

\begin{rem}
Theorem \ref{thm:algebraicity_of_stack_of_stable_maps} remains true if we replace $X$ by an algebraic space locally of finite presentation over $S$. The proof carries over using analogous theories for algebraic spaces.
\end{rem}

\begin{rem}
If we assume moreover that the target space $X$ is separated over $S$, then by the valuative criterion of separatedness, the algebraic stack $\bcMgn(X/S)$ is also separated over $S$ (cf.\ \cite{Kontsevich_Enumeration_1995} \S 1.3.1).
\end{rem}

\begin{rem}
In characteristic zero, the stability condition (Definition \ref{def:stable_maps_(Gromov)}(iv)) ensures that a stable map has no infinitesimal automorphisms. So $\bcMgn(X/S)$ is a Deligne-Mumford stack in characteristic zero. In general, one can only ensure that the diagonal morphism of $\bcMgn(X/S)$ is quasi-finite.
\end{rem}

\section{The moduli stack of non-archimedean analytic stable maps}\label{sec:stacks_of_k-analytic_stable_maps}

In this section, we construct the moduli stack of formal stable maps and the moduli stack of non-archimedean analytic stable maps.

Definitions \ref{def:stable_maps_(Gromov)} and \ref{def:stable_curves} can be carried verbatim to the cases of formal schemes and $k$-analytic spaces as follows.

\begin{defin}\label{def:stable_maps_new}
Let $X$, $T$ be formal schemes over $k^\circ$ (resp.\ $k$-analytic spaces). An \emph{\gn formal (resp.\ $k$-analytic) stable map} $\big(C\rightarrow T, (s_i), f\big)$ into $X$ over $T$ consists of a morphism $C\rightarrow T$, a morphism $f\colon C\rightarrow X$ and $n$ morphisms $s_i\colon T\rightarrow C$ such that Conditions (i)-(iv) of Definition \ref{def:stable_maps_(Gromov)} are satisfied.
\end{defin}

\begin{defin}\label{def:stable_curves_new}
Let $T$ be a formal scheme over $k^\circ$ (resp.\ a $k$-analytic space). An \emph{\gn formal (resp.\ $k$-analytic) stable curve} $\big(C\rightarrow T, (s_i)\big)$ over $T$ is an \gn formal (resp.\ $k$-analytic) stable map $\big(C\rightarrow T, (s_i), f\big)$ into $\Spf k^\circ$ (resp.\ $\SpB k$) with $f$ being the structure morphism.
\end{defin}

These definitions are compatible with respect to the analytification functor $(\cdot)\an$, the special fiber functor $(\cdot)_s$ and the generic fiber functor $(\cdot)_\eta$ in the following sense.

\begin{lem}\label{lem:analytification_of_stable_maps}
Let $X$ be an algebraic variety over $k$, $A$ a strictly $k$-affinoid algebra and $\big(C\rightarrow\Spec A, (s_i),f\big)$ an \gn algebraic stable map into $X$ over $\Spec A$. Then the triple $\big(C\an\rightarrow\SpB A, (s\an_i),f\an\big)$ obtained by applying the relative analytification functor $(\cdot)\an$ is an \gn $k$-analytic stable map into $X\an$ over $\SpB A$.
\end{lem}
\begin{proof}
Conditions (i)-(iii) of Definition \ref{def:stable_maps_(Gromov)} are clearly satisfied.
Condition (iv) follows from the fact that a geometric point of the $k$-analytic spectrum $\SpB A$ is in particular a geometric point of $\Spec A$.
\end{proof}

\begin{lem}
Let $\fX,\fT$ be formal schemes locally finitely presented over $k^\circ$ and let $\big(\fC\rightarrow\fT,(\fs_i),\ff\big)$ be an \gn formal stable map into $\fX$ over $\fT$.
Then the triple $\big(\fC_s\rightarrow \fT_s,((\fs_i)_s),\ff_s\big)$ obtained by applying the special fiber functor $(\cdot)_s$ is an \gn algebraic stable map into $\fXs$ over $\fT_s$.
\end{lem}
\begin{proof}
Conditions (i)-(iii) of Definition \ref{def:stable_maps_(Gromov)} are clearly satisfied.
Condition (iv) follows from the fact that a geometric point of the scheme $\fT_s$ is in particular a geometric point of the formal scheme $\fT$.
\end{proof}

\begin{lem}\label{lem:generic_fiber_of_formal_stable_maps}
Let $\fX,\fT$ be formal schemes locally finitely presented over $k^\circ$ and let $\big(\fC\rightarrow\fT,(\fs_i),\ff\big)$ be an \gn formal stable map into $\fX$ over $\fT$.
Then the triple $\big(\fC_\eta\rightarrow \fT_\eta,((\fs_i)_\eta),\ff_\eta\big)$ obtained by applying the generic fiber functor $(\cdot)_\eta$ is an \gn $k$-analytic stable map into $\fXe$ over $\fT_\eta$.
\end{lem}
\begin{proof}
Conditions (i)-(iii) of Definition \ref{def:stable_maps_(Gromov)} are clearly satisfied.
Let us verify Condition (iv).
A geometric point of $\fT_\eta$ is a morphism $\SpB k'\rightarrow\fT_\eta$ for some algebraically closed non-archimedean field $k'$.
Denote by $(k')^\circ$ the ring of integers of $k'$ and let $\fT'=\Spf(k')^\circ$.
Let $\big(\fC'\rightarrow \fT',(\fs'_i),\ff'\big)$ be the pullback of the formal stable map $\big(\fC\rightarrow\fT,(\fs_i),\ff\big)$ along the morphism $\fT'\rightarrow\fT$. It suffices to show that the triple $\big(\fC'_\eta\rightarrow \fT'_\eta,((\fs'_i)_\eta),\ff'_\eta\big)$ obtained by applying the generic fiber functor $(\cdot)_\eta$ satisfies the stability condition.

Denote by $\big(\fC'_s\rightarrow \fT'_s,((\fs'_i)_s),\ff'_s\big)$ the \gn algebraic stable map into $\fXs$ over $\Spec\widetilde{k'}$ obtained by applying the special fiber functor $(\cdot)_s$.
Let $\pi_{\fC'}\colon \fC'_\eta\rightarrow\fC'_s$ denote the reduction map.
The triple $\big(\fC'_\eta\rightarrow \fT'_\eta,((\fs'_i)_\eta),\ff'_\eta\big)$ can be seen as an infinitesimal deformation of $\big(\fC'_s\rightarrow \fT'_s,((\fs'_i)_s),\ff'_s\big)$.
Let $C^0$ be an irreducible component of $\fC'_\eta$ that maps to a point by $\ff'_\eta$.
Then $C^0$ is an infinitesimal deformation of $\pi_{\fC'}(C^0)$.
Since $\big(\fC'_s\rightarrow \fT'_s,((\fs'_i)_s),\ff'_s\big)$ is an algebraic stable map, every irreducible component of $\pi_{\fC'}(C^0)$ of genus 0 (resp.\ 1) has at least 3 (resp.\ 1) special points.
We deduce that the irreducible component $C^0$ also has enough special points, completing the proof.
\end{proof}

For a formal scheme $\fX$ over $k^\circ$, we denote by $\bcMgn(\fX)$ the moduli stack of \gn formal stable maps into $\fX$.
For a $k$-analytic space $X$, we denote by $\bcMgn(X)$ the moduli stack of \gn $k$-analytic stable maps into $X$.
In the following, we study the representability of the two stacks.

Let $\varpi$ be a uniformizer of the ground field $k$.
Put $S_m=\Spec (k^\circ/\varpi^{m+1})$.

\begin{prop}\label{prop:moduli_of_formal_stable_maps}
Let $\fX$ be a formal scheme locally finitely presented over $k^\circ$.
Put $X_m=\fX\times_{k^\circ} S_m$.
We have a natural isomorphism of stacks over the category $\FSch_{\kc}$
\[ \varinjlim_m \bcMgn(X_m/S_m) \xrightarrow{\ \sim\ } \bcMgn(\fX). \]
Therefore, the stack $\bcMgn(\fX)$ is a formal stack locally finitely presented over $k^\circ$.
Consequently, we have a natural isomorphism
\[\big(\bcMgn(\fX)\big)_s\xrightarrow{\ \sim\ }\bcMgn(\fXs),\]
where $(\cdot)_s$ denote the special fiber functor.
\end{prop}
\begin{proof}
By \cref{thm:algebraicity_of_stack_of_stable_maps}, the stack $\bcMgn(X_m/S_m)$ is an algebraic stack locally finitely presented over $S_m$.
For a formal scheme $\fT\in\FSch_{\kc}$, we denote $T_m = \fT\times_{k^\circ} S_m$.
A morphism $\fT\rightarrow\varinjlim_m \bcMgn(X_m/S_m)$ corresponds to a compatible sequence of morphisms $t_m\colon T_m\rightarrow \bcMgn(X_m/S_m)$ in the sense that $t_m = t_{m+1}\times_{S_{m+1}} S_m$.
The latter corresponds to a compatible sequence of \gn algebraic stable maps into $X_m$ over $T_m$, and thus an \gn formal stable map into $\fX$ over $\fT$.
So we deduce the natural isomorphism $\varinjlim_m \bcMgn(X_m/S_m) \xrightarrow{\sim} \bcMgn(\fX)$.
\end{proof}

When a $k$-analytic space is the analytification of a proper algebraic variety over $k$, the representability of the moduli stack of \kanal stable maps can be established as the analytification of the moduli stack of algebraic stable maps using the non-archimedean analytic GAGA theorem.

\begin{thm}\label{thm:moduli_of_analytic_stable_maps_GAGA}
Let $X$ be a proper algebraic variety over $k$. We have a natural isomorphism of stacks over the category $\An_k$
\[ \big(\bcMgn(X)\big)\an\xrightarrow{\ \sim\ }\bcMgn(X\an), \]
where $(\cdot)\an$ denotes the analytification functor.
Therefore, the stack $\bcMgn(X\an)$ is a $k$-analytic stack.
\end{thm}
\begin{proof}
Let $T=\SpB A$ for any strictly $k$-affinoid algebra $A$.
A morphism $T\to\big(\bcMgn(X)\big)\an$ gives rise to an \gn algebraic stable map into $X$ over $\Spec A$.
By Lemma \ref{lem:analytification_of_stable_maps}, the analytification of this algebraic stable map gives a $k$-analytic stable map into $X\an$ over $T$.
So we obtain a morphism $T\to\bcMgn(X\an)$.
Since the construction is functorial on $T$, we obtain a natural morphism $\big(\bcMgn(X)\big)\an\xrightarrow{\ \sim\ }\bcMgn(X\an)$.

Now we show that the functor
\begin{equation}\label{eq:moduli_of_analytic_stable_maps_GAGA}
\big(\bcMgn(X)\big)\an(T) \xrightarrow{\ \sim\ } \bcMgn(X\an)(T)
\end{equation}
is an equivalence of groupoids.
It is faithful by construction.
Let us prove that it is essentially surjective.
Let $\big(C\rightarrow T, (s_i),f\big)$ be an \gn $k$-analytic stable map into $X\an$ over $T$. Denote by $S_i$ the image of the section $s_i$.
Put $X\an_T=X\an\times_k T$, $f_T\colon C\rightarrow X\an_T$.
The stability condition implies that the line bundle $\omega_{C/T}(\sum S_i)$ is relatively ample over $X\an_T$.
Let $\mathcal L$ denote the sheaf of graded algebras
\[ \mathcal L = \bigoplus_{m\geq 0} f_{T*} \Big(\omega_{C/T} \big(\sum S_i\big)\Big)^{\otimes m} \]
over $X\an_T$.
It is a direct sum of $k$-analytic coherent sheaves.
By the relative $k$-analytic GAGA theorem (cf.\ \cite{Conrad_Relative_2006} Example 3.2.6, \cite{Kopf_Eigentliche_1974} \S 5-\S 6), 
$\mathcal L$ is isomorphic to the analytification of an algebraic quasi-coherent sheaf of graded algebras $\mathcal L\alg$ over $X_T$. Using the relative $\Proj$ construction, the sheaf $\mathcal L\alg$ gives us a family of algebraic curves $C\alg=\Proj_{X_T} (\mathcal L\alg)$ over $\Spec A$ and a morphism from $C\alg$ to $X$.
The algebraization of the sections $(s_i)$ follows from the same GAGA theorem.
So we obtain an \gn algebraic stable map into $X$ over $\Spec A$, whose analytification is the $k$-analytic stable map we started with.
Therefore, we have proved that the functor \eqref{eq:moduli_of_analytic_stable_maps_GAGA} is essentially surjective.
Similarly, the fullness of the functor \eqref{eq:moduli_of_analytic_stable_maps_GAGA} follows from the GAGA theorem as well.
We conclude that the functor \eqref{eq:moduli_of_analytic_stable_maps_GAGA} is an equivalence of groupoids.
\end{proof}

\begin{cor}\label{cor:moduli_of_stable_curves}
The moduli stack of \kanal stable curves is the analytification of the moduli stack of algebraic stable curves over $k$.
\end{cor}

In order to establish the representability of the moduli stack of non-archimedean analytic stable maps in general, we need to resort to the moduli stack of formal stable maps, whose representability has been established in Proposition \ref{prop:moduli_of_formal_stable_maps}.

\begin{thm}\label{thm:moduli_of_analytic_stable_maps}
Let $\fX$ be a formal scheme locally of finite presentation over $k^\circ$.
We have a natural isomorphism of stacks over the category $\An_k$
\[ \big(\bcMgn(\fX)\big)_\eta \xrightarrow{\ \sim\ } \bcMgn(\fXe), \]
where $(\cdot)_\eta$ denotes the generic fiber functor. Therefore, the stack $\bcMgn(\fXe)$ is a $k$-analytic stack.
\end{thm}

\begin{cor}\label{cor:moduli_of_analytic_stable_maps}
Let $X$ be a paracompact strictly $k$-analytic space.
Then the stack $\bcMgn(X)$ of \gn $k$-analytic stable maps into $X$ is a $k$-analytic stack.
\end{cor}
\begin{proof}
By the theory of formal models (cf.\ \cite{Raynaud_Geometrie_analytique_rigide_1974,Bosch_Formal_I}), there exists a formal scheme $\fX$ locally of finite presentation over $\kc$ whose generic fiber $\fXe$ is isomorphic to the $k$-analytic space $X$.
Therefore, the corollary follows from \cref{thm:moduli_of_analytic_stable_maps}.
\end{proof}

For the proof of \cref{thm:moduli_of_analytic_stable_maps}, we need to establish the existence of formal models for $k$-analytic stable maps.

\begin{thm}\label{thm:formal_model_stable_maps}
Let $\fX$ be a formal scheme locally of finite presentation over $k^\circ$.
Let $T$ be a strictly $k$-affinoid space and let $\big(C\rightarrow T,(s_i),f\big)$ be an \gn \kanal stable map into $\fXe$ over $T$.
Up to passing to a quasi-étale covering of $T$, there exists a formal model $\fT$ of $T$ and an \gn formal stable map into $\fX$ over $\fT$ that gives back the $k$-analytic stable map when we apply the generic fiber functor.
\end{thm}

\begin{proof}
By the theory of formal models (cf.\ \cite{Raynaud_Geometrie_analytique_rigide_1974,Bosch_Formal_I}), we can obtain a formal model $\fC$ of $C$ and a formal model $\fT$ of $T$ such that the morphisms of $k$-analytic spaces $C\rightarrow T$, $f\colon C\rightarrow\fXe$ and $s_i\colon T\rightarrow C$ extend to morphisms of formal schemes $\fC\rightarrow\fT$, $\ff\colon \fC\rightarrow\fX$ and $\fs_i\colon \fT\rightarrow\fC$.
Our aim is to modify the triple $\big(\fC\rightarrow\fT,(\fs_i),\ff\big)$ in order to obtain a formal stable map into $\fX$.
Our proof uses ideas from the works of de Jong's on alterations \cite{de_Jong_Smoothness_1996,de_Jong_Families_1997}.
For clarity, we decompose our reasonings into 6 steps.

\smallskip
\emph{Step 1 (algebraize).}
Let $\cA$ be the topological algebra finitely presented over $\kc$ such that $\fT=\Spf\cA$.
Put $\fT\alg=\Spec\cA$.
Since $\fC\rightarrow\fT$ is a family of proper formal curves, up to passing to a Zariski covering of $\fT$, we can assume that the morphism $\fC\rightarrow\fT$ is projective.
By formal GAGA (cf.\ \cite{EGA3-1} \S 5), the morphism $\fC\rightarrow\fT$ is isomorphic to the completion of a family of algebraic curves $\fC\alg\rightarrow\fT\alg$ along the special fibers over the residue field $\tilde k$.
By formal GAGA again, we can algebraize the sections $(\fs_i)$ as well.
So we obtain a family of pointed algebraic curves $\big(\fC\alg\rightarrow\fT\alg,(\fs_i)\big)$.

\smallskip
\emph{Step 2 (rigidify).}
Up to passing to an étale covering of $\fT\alg$, we can add extra disjoint sections $(\fs'_j)$ such that every irreducible component of every geometric fiber of the morphism $\fC\alg\rightarrow\fT\alg$ meets at least three sections. Denote by $n'$ the total number of sections.

\smallskip
\emph{Step 3 (kill monodromies).}
If we pass to the generic fibers of the family of pointed algebraic curves $\big(\fC\alg\rightarrow\fT\alg,(\fs_i)\cup(\fs'_j)\big)$, i.e.\ take fiber product with $\Spec k$ over $\Spec k^\circ$, we obtain an \gnprime algebraic stable curve over $\fT_k\alg\coloneqq\fT\alg\times_{\kc} k$.
This gives rise to a morphism $t_C$ from $\fT_k\alg$ to the Deligne-Mumford stack $\bcMgnprime$ of \gnprime algebraic stable curves over $k$.
Let $M^0\rightarrow\bcMgnprime$ be a representable étale atlas of the stack.
Put $\fT^{\mathrm{alg},0}_k=\fT_k\alg\times_{\bcMgnprime} M^0$.
Denote by $t^0_C$ the morphism $\fT^{\mathrm{alg},0}_k\rightarrow M^0$ obtained from the fiber product.
We obtain an \gnprime algebraic stable curve over $\fT^{\mathrm{alg},0}_k$ by pulling back the universal family over $\bcMgnprime$.

\smallskip
\emph{Step 4 (extend families of stable curves).}
By the properness of the Deligne-Mumford stack $\bcMgnprime$, one can assume that the scheme $M^0$ is of finite type over $k$.
Let $\overline M^0$ be the normalization of $\bcMgnprime$ in the total ring of fractions of $M^0$.
Then $\overline M^0$ is a proper scheme by the properness of $\bcMgnprime$.
Let $\widehat\fT\alg$ be the closure of the image of the morphism $\fT^{\mathrm{alg},0}_k\rightarrow \fT\alg\times\overline M^0$.
We obtain a proper morphism $\widehat \fT\alg\rightarrow\fT\alg$ and a morphism $\widehat \fT\alg\rightarrow\overline M^0$.
The latter morphism gives rise to an \gnprime algebraic stable curve $\widehat\fC\alg$ over $\widehat\fT\alg$ by pulling back the universal family over $\bcMgnprime$.

\smallskip
\emph{Step 5 (extend morphisms of stable curves).} Put $\overline\fC\alg=\fC\alg\times_{\fT\alg}\widehat\fT\alg$.
By construction, we have a morphism $\widehat\fC\alg\times_{\kc} k\rightarrow\overline\fC\alg\times_{\kc} k$ over $\widehat\fT\alg\times_{\kc} k$.
We would like to extend it to a morphism from $\widehat\fC\alg$ to $\overline\fC\alg$ over $\widehat\fT\alg$.
Consider the Hom scheme $\mathit{Mor}\coloneqq\Mor_{\widehat\fT\alg}\big(\widehat\fC\alg,\overline\fC\alg\big)$ parameterizing morphisms from $\widehat\fC\alg$ to $\overline\fC\alg$ over $\widehat\fT\alg$ (see \cite{Mumford_Geometric_1994}).
The morphism $\widehat\fC\alg\times_{\kc} k\rightarrow\overline\fC\alg\times_{\kc} k$ corresponds to a morphism $t_M\colon \widehat\fT\alg\times_{\kc} k\rightarrow \mathit{Mor}$. Let $\widehat{\widehat\fT\alg}$ be the closure of the image of $t_M$ in the scheme $\mathit{Mor}$.
We claim that the morphism $\widehat{\widehat\fT\alg}\rightarrow\widehat\fT\alg$ is proper.
Indeed, the properness can be verified using the valuative criterion.
Let $T'$ be a trait, i.e.\ the spectrum of a complete discrete valuation ring.
Let $T'\rightarrow\widehat\fT\alg$ be a morphism that sends the generic point $\xi$ of $T'$ to the generic fiber $\widehat\fT\alg\times_{\kc} k$ of $\widehat\fT\alg$.
Let $\widehat C'\rightarrow T'$ and $\overline C'\rightarrow T'$ denote the pullback of $\widehat\fC\alg$ and $\overline\fC\alg$ respectively along the morphism $T'\rightarrow\widehat\fT\alg$.
We have a morphism $\widehat C'_\xi\rightarrow\overline C'_\xi$ over the generic point $\xi$ of $T'$ compatible with the sections.
Moreover, every irreducible component of every geometric fiber of the morphism $\overline C'\rightarrow T'$ has at least three marked points by Step 2.
In the case where the generic fiber $\widehat C'_\xi$ is smooth, the total space $\widehat C'$ is regular.
Let $G$ denote the closure of the graph of the morphism $\widehat C'_\xi\rightarrow\overline C'_\xi$ inside $\widehat C'\times_{T'} \overline C'$.
We can apply de Jong's three point lemma to the normalizations of $G$ and $\overline C'$ (cf.\ \cite{de_Jong_Smoothness_1996} \S 4.18, \cite{Abramovich_Alterations_2000} \S 4.9, \S 5.1, \cite{Abramovich_Stable_2001}).
So the morphism $\widehat C'_\xi\rightarrow\overline C'_\xi$ can be extended to a morphism $\widehat C'\rightarrow\overline C'$ thanks to the three point condition.
For the general case, by the deformation theory of pointed stable curves, any double point on the generic fiber $\widehat C'_\xi$ comes from a double point on the special fiber $\widehat C'_s$. We can make a finite base change of $T'$ so that all the double points split. Denote by $\widetilde{\widehat C'}$ the normalization of $\widehat C'$. Then the generic fiber $\big(\widetilde{\widehat C'}\big)_\xi$ is smooth. Applying the argument above to $\widetilde{\widehat C'}$, the morphism $\big(\widetilde{\widehat C'}\big)_\xi\rightarrow\overline C'_\xi$ over $T'_\xi$ extends to a morphism $\widetilde{\widehat C'}\rightarrow\overline C'$ over $T'$. It descends to a morphism $\widehat C'\rightarrow\overline C'$ over $T'$ by continuity.
By the valuative criterion of properness, we conclude that the morphism $\widehat{\widehat\fT\alg}\rightarrow\widehat\fT\alg$ is proper.
Therefore, by replacing $\widehat\fT\alg$ with $\widehat{\widehat\fT\alg}$, and by replacing $\widehat\fC\alg$ and $\overline\fC\alg$ with their pullbacks, the morphism $\widehat\fC\alg\times_{\kc} k\rightarrow\overline\fC\alg\times_{\kc} k$ can be extended to a morphism $\widehat\fC\alg\rightarrow\overline\fC\alg$.

\smallskip
\emph{Step 6 (complete).} Let $\widehat\fT$, $\widehat\fC$, $\overline\fC$ be the completions of the schemes $\widehat\fT\alg$, $\widehat\fC\alg$, $\overline\fC\alg$ respectively along their special fibers over the residue field $\tilde k$.
Then $\big(\widehat\fC\rightarrow\widehat\fT, (\fs_i)\cup(\fs'_j)\big)$ is a formal stable curve, where $(\fs'_j)$ denote the extra sections introduced in Step 2.
Let $\widehat\ff$ be the composition of the morphisms
\[\widehat\fC\rightarrow\overline\fC\rightarrow\fC\xrightarrow{\ff}\fX .\]
If we take into account the extra sections $(\fs'_j)$, we have constructed a formal stable map $\big(\widehat\fC\rightarrow\widehat\fT, (\fs_i)\cup(\fs'_j),\widehat\ff\big)$ into $\fX$ over $\widehat\fT$.
Now we remove these extra sections and then contract the non-stable components,
we obtain an \gn formal stable map into $\fX$ over $\widehat\fT$ which we denote by $\Big(\overline{\widehat\fC}\rightarrow\widehat\fT,(\fs_i),\overline{\widehat\ff}\Big)$.
If we apply the generic fiber functor $(\cdot)_\eta$ to it, we get back the $k$-analytic stable map $\big(C\rightarrow T,(s_i),f\big)$ we started with, up to a base change by a quasi-étale covering of $T$.
\end{proof}

\begin{proof}[Proof of \cref{thm:moduli_of_analytic_stable_maps}]
Let $\fT$ be an affine formal scheme finitely presented over $\kc$.
A morphism $\fT\to\bcMgn(\fX)$ give rise to an \gn formal stable map into $\fX$ over $\fT$.
By \cref{lem:generic_fiber_of_formal_stable_maps}, applying the generic fiber functor, we obtain an \gn \kanal stable map into $\fXe$ over $\fT_\eta$.
So we obtain a morphism $\fT_\eta\to\bcMgn(\fXe)$.
This construction gives rise to a natural morphism $\big(\bcMgn(\fX)\big)_\eta\to\bcMgn(\fXe)$.

Now we show that for any strictly $k$-affinoid space $T$, the functor
\begin{equation}\label{eq:moduli_of_analytic_stable_maps}
\big(\bcMgn(\fX)\big)_\eta(T) \xrightarrow{\ \sim\ } \bcMgn(\fXe)(T)
\end{equation}
is an equivalence of groupoids.
The functor is faithful by construction.
Let us prove that it is full.
Let $\big(\fC_1\rightarrow\fT_1, (\fs_{i,1}), \ff_1\big)$
and $\big(\fC_2\rightarrow\fT_2, (\fs_{i,2}), \ff_2\big)$ be two \gn formal stable maps into $\fX$.
Assume that when we pass to the generic fibers, we have an isomorphism of the $k$-analytic stable maps
\begin{equation}\label{eq:extend_isom}
\big((\fC_1)_\eta\rightarrow(\fT_1)_\eta, ((\fs_{i,1})_\eta), (\ff_1)_\eta\big)\xrightarrow{\ \sim\ }
\big((\fC_2)_\eta\rightarrow(\fT_2)_\eta, ((\fs_{i,2})_\eta), (\ff_2)_\eta\big).
\end{equation}
We can assume that the formal schemes $\fT_1$ and $\fT_2$ are flat over $k^\circ$ by killing the torsions.
Up to replacing $\fT_1$ and $\fT_2$ by admissible blowups, we can assume that $\fT_1\simeq\fT_2$, and we denote them by $\fT_{12}$.
Up to passing to a Zariski covering of $\fT_{12}$, using the formal GAGA theorem as in the proof of \cref{thm:formal_model_stable_maps}, we can assume that the pointed formal pre-stable curves
$\big(\fC_1\rightarrow\fT_{12}, (\fs_{i,1})\big)$
and $\big(\fC_2\rightarrow\fT_{12}, (\fs_{i,2})\big)$ are completions of pointed algebraic pre-stable curves, which we denote by
$\big(\fC_1\alg\rightarrow\fT_{12}\alg, (\fs_{i,1})\big)$
and $\big(\fC_2\alg\rightarrow\fT_{12}\alg, (\fs_{i,2})\big)$ respectively.
Up to passing to an étale covering of $\fT_{12}\alg$, we can add extra sections to the pointed algebraic pre-stable curves above to obtain two pointed algebraic stable curves $\big(\fC_1\alg\rightarrow\fT_{12}\alg, (\fs_{i,1})\cup(\fs'_{j,1})\big)$ and $\big(\fC_2\alg\rightarrow\fT_{12}\alg, (\fs_{i,2})\cup(\fs'_{j,2})\big)$, such that $(\fs'_{j,1})_\eta\simeq(\fs'_{j,2})_\eta$ for every $j$.
Consider the $\Isom$ scheme parameterizing isomorphisms between the algebraic stable curves
\[\mathit{Isom} \coloneqq \Isom_{\fT_{12}\alg}\Big(\big(\fC_1\alg\rightarrow\fT_{12}\alg,(\fs_{i,1})\cup(\fs'_{j,1})\big), \big(\fC_2\alg\rightarrow\fT_{12}\alg, (\fs_{i,2})\cup(\fs'_{j,2})\big)\Big).\]
We have a morphism $t_I:\fT_{12}\alg\times_{\kc} k\rightarrow \mathit{Isom}$ induced by the isomorphism over the generic fibers.
Let $\widehat{\fT_{12}\alg}$ be the closure of the image of $t_I$ in the scheme $\mathit{Isom}$.
Using the separatedness of the moduli space of pointed algebraic stable curves and the valuative criterion of properness, we prove that the morphism $\widehat{\fT_{12}\alg}\rightarrow \fT_{12}\alg$ is proper.
Let $\widehat{\fT_{12}}$ denote the completion of $\widehat{\fT_{12}\alg}$ along its special fiber over the residue field $\tilde k$.
We deduce that the isomorphism of the $k$-analytic pre-stable curves
\[\big((\fC_1)_\eta\rightarrow(\fT_1)_\eta, ((\fs_{i,1})_\eta)\big)\xrightarrow{\ \sim\ }
\big((\fC_2)_\eta\rightarrow(\fT_2)_\eta, ((\fs_{i,2})_\eta)\big)\]
in \eqref{eq:extend_isom} can be extended over the formal scheme $\widehat{\fT_{12}}$.
The isomorphism $(\ff_1)_\eta\simeq(\ff_2)_\eta$ also extends by continuity.
This implies the fullness of the functor \eqref{eq:moduli_of_analytic_stable_maps}.

In order to show that the functor \eqref{eq:moduli_of_analytic_stable_maps} is essentially surjective, we start with a morphism $T\to\bcMgn(\fXe)$.
This gives rise to an \gn $k$-analytic stable map $\big(C\rightarrow T,(s_i),f\big)$ into $\fXe$ over $T$.
By \cref{thm:formal_model_stable_maps}, there exists a quasi-étale covering $T^0\to T$, a formal model $\fT^0$ of $T^0$ and an \gn formal stable map into $\fX$ over $\fT^0$ that gives the pullback to $T^0$ of the $k$-analytic stable map $\big(C\rightarrow T,(s_i),f\big)$ when we apply the generic fiber functor.
So we obtain a morphism $\fT^0\to\bcMgn(\fX)$ and a morphism $T^0\to\big(\bcMgn(\fX)\big)_\eta$.
Using the proof of the fullness of the functor \eqref{eq:moduli_of_analytic_stable_maps}, we see that the morphism $T^0\to\big(\bcMgn(\fX)\big)_\eta$ descends to a morphism $T\to\big(\bcMgn(\fX)\big)_\eta$, which corresponds to the morphism $T\to\bcMgn(\fXe)$ we started with.
We conclude that the functor \eqref{eq:moduli_of_analytic_stable_maps} is an equivalence of groupoids, completing the proof of the theorem.
\end{proof}

\section{Gromov compactness in non-archimedean analytic geometry} \label{sec:non-archimedean_Gromov_compactness}

In this section, we prove the non-archimedean Gromov compactness theorem.

Let $X$ be a $k$-analytic space, and let $\hL$ be a Kähler structure on $X$ with respect to an SNC formal model $\fX$ of $X$ (Definition \ref{def:kahler}).
Let $C$ be a connected proper smooth $k$-analytic curve and let $f\colon C\rightarrow X$ be a morphism.

The degree of the morphism $f$ with respect to $\hL$ is defined as follows.
We choose an SNC formal model $\fC$ of $C$ and a morphism $\ff\colon \fC\rightarrow\fX$ such that $\ff_\eta\simeq f$.
The pullback $\ff^*L$ is a virtual line bundle on $C$ with respect to the formal model $\fC$. We define $\deg f = \deg \ff^*L$ (cf.\ Definition \ref{def:degree_of_virtual_line_bundle}).
Since two different choices of the formal model $\fC$ can always be dominated by another one, one can show that the degree $\deg f$ does not depend on the choice.
If the $k$-analytic curve $C$ is not smooth but with ordinary double points, we define the degree of the morphism $f$ to be the sum of the degree for every connected component after taking normalization. 

Similarly, for a morphism $f\colon \cC\rightarrow\fXs$ from an algebraic pre-stable curve $\cC$ over the residue field $\tilde k$ to $\fXs$, we define $\deg f=\deg f^* c(\hL)$ (cf.\ Definition \ref{def:degree_of_curvature}).
For a morphism $\ff\colon \fC\rightarrow\fX$ from a formal pre-stable curve $\fC$ over $k^\circ$ to $\fX$, we define $\deg \ff = \deg \ff_s$.

Fix a positive real number $A$. Denote by $\overline\cM_{g,n} (X,A)$ the moduli stack of \gn $k$-analytic stable maps into $X$ whose degree with respect to $\hL$ is bounded by $A$. Similarly, we define the moduli stacks $\overline\cM_{g,n} (\fX,A)$ and $\overline\cM_{g,n} (\fXs,A)$.
Proposition \ref{prop:moduli_of_formal_stable_maps} and Theorem \ref{thm:moduli_of_analytic_stable_maps} imply the following isomorphisms
\begin{align*}
\overline\cM_{g,n} (\fX,A)_s &\simeq \overline\cM_{g,n} (\fXs,A),\\
\overline\cM_{g,n} (\fX,A)_\eta &\simeq \overline\cM_{g,n} (X,A).
\end{align*}

\begin{prop}\label{prop:properness_of_special_fiber}
The moduli stack $\bcMgn(\fXs,A)$ is an algebraic stack of finite type over the residue field $\tilde k$. It is a proper algebraic stack if the formal model $\fX$ is proper.
\end{prop}
\begin{proof}
Let $\big\{\,\tD_i\;\big|\;i\in\IX\,\big\}$ denote the set of the irreducible components of the special fiber $\fXs$.
The Kähler structure $\widehat L$ implies that every $\tD_i$ is quasi-projective.
It follows from the boundedness of Hilbert schemes that there exists an integer $N^0_i$ such that for any stable map $\big(C\rightarrow\Spec\tilde k,(s_i),f\big)$ into $\tD_i$ of genus bounded by $g$ and degree bounded by $A$, the number of irreducible components of $C$ is bounded by $N_i^0$.

Given a stable map $F\coloneqq\big(C\rightarrow\Spec\tilde k,(s_i),f\big)$ into $\fXs$, by normalizing certain double points and adding new marked points to the previous double points, we can decompose $F$ into finitely many stable maps $F_{ij}$ for $i\in\IX$, $j\in\{1,\dots,N_i\}$, $N_i\ge 0$, such that the stable map $F_{ij}$ maps into the component $\tD_i$.
Note that this decomposition is not necessarily unique.
If given such a decomposition, we let $g_{ij}$ denote the genus of the stable map $F_{ij}$, let $n_{ij}$ denote the number of marked points on $F_{ij}$ that come from the marked points on $F$, and let $n_{ij}^{i'j'}$ denote the number of double points $\nu$ on $F$ such that in the decomposition, one branch of $\nu$ lies in $F_{ij}$ and the other branch lies in $F_{i'j'}$.

Motivated by the consideration above, we introduce the following definition.

\begin{defin}
A \emph{decomposition datum} $\cD$ is a collection of the following non-negative integers:
\begin{itemize}
\item $N_i$ for all $i\in\IX$, such that $N_i\leq N_i^0$,
\item $g_{ij}, n_{ij}$ for all $i\in\IX$, $j\in\{1,\dots,N_i\}$,
\item $n_{ij}^{i'j'}$ for all $i,i'\in\IX$, $j\in\{1,\dots,N_i\}$, $j'\in\{1,\dots,N_{i'}\}$, such that $n_{ij}^{i'j'}=0$ whenever $i=i'$, and $n_{ij}^{i'j'} = n_{i'j'}^{ij}$ for all $i,j,i',j'$.
\end{itemize}
\end{defin}

Given a decomposition datum $\cD$, we associate to it a graph $G_\cD$ as follows. The vertices of $G_\cD$ are labeled by pairs of integers $(i,j)$ for all $i\in\IX$, $j\in\{1,\dots,N_i\}$. The number of edges between the vertex $(i,j)$ and the vertex $(i',j')$ is given by $n_{ij}^{i'j'}$. Let $b_1(G_\cD)$ denote the first Betti number of the graph $G_\cD$.

\begin{defin}
A decomposition datum $\cD$ is said to be of type $(g,n)$ if
\begin{enumerate}[(i)]
\item the associated graph $G_\cD$ is connected,
\item $b_1(G_\cD)+\sum_{i,j} g_{ij} = g$, and
\item $\sum_{i,j} n_{ij}=n$.
\end{enumerate}
\end{defin}

It follows from the definitions that

\begin{lem}
There is only a finite number of decomposition data of type $(g,n)$.
\end{lem}

Given a decomposition datum $\cD$ of type $(g,n)$, we consider the following construction. For every $i\in\IX$, $j\in\{1,\dots,N_i\}$, put $n'_{ij} = n_{ij}+\sum_{i',j'} n_{ij}^{i'j'}$.
Let $\bcMgnijprime(\tD_i,A)$ denote the moduli stack of $n'_{ij}$-pointed genus $g_{ij}$ stable maps into the irreducible component $\tD_i$ of degree bounded by $A$.
The Kähler structure $\widehat L$ implies that the irreducible component $\tD_i$ is quasi-projective.
If the formal model $\fX$ is proper, then $\tD_i$ is projective.
In this case, $\bcMgnijprime(\tD_i,A)$ is a proper algebraic stack over the residue field $\tilde k$ by Theorem \ref{thm:stack_of_stable_maps_projective_case}.
If we don't assume the properness of $\fX$, then $\bcMgnijprime(\tD_i,A)$ is an open sub-stack of a proper algebraic stack.
So it is an algebraic stack of finite type over $\tilde k$.

Let us label the $n'_{ij}$ marked points as follows:
\begin{multline*}
(1)_{ij}, (2)_{ij},\dots,(n_{ij})_{ij},\\
\text{and } (1)_{ij}^{i'j'}, (2)_{ij}^{i'j'},\dots,(n_{ij}^{i'j'})_{ij}^{i'j'}\text{ for all }i'\in\IX, j'\in\{i,\dots,N_{i'}\}.
\end{multline*}

Let $\bcMgn(\fXs,A)_\text{dec}$ denote the moduli stack of collections of $n'_{ij}$-pointed genus $g_{ij}$ stable maps into $\tD_i$ for every $i\in\IX$, $j\in\{1,\dots N_i\}$ which satisfy the following conditions:
\begin{enumerate}[(i)]
\item The marked point with label $(l)_{ij}^{i'j'}$ maps into the stratum $\tD_{\{i,i'\}}=\tD_i\cap \tD_{i'}$.
\item The marked point with label $(l)_{ij}^{i'j'}$ and the marked point with label $(l)^{ij}_{i'j'}$ map to the same point in the stratum $\tD_{\{i,i'\}}$.
\item If we denote by $A_{ij}$ the degree of the $n'_{ij}$-pointed genus $g_{ij}$ stable map into $\tD_i$ with respect to $\hL$, then we have $\sum A_{ij}\leq A$.
\end{enumerate}

The morphism
\[\bcMgn(\fXs,A)_\text{dec}\rightarrow\prod_{i,j} \bcMgnijprime(\tD_i,A)\]
is a closed embedding of finite presentation. Therefore, it is an algebraic stack of finite type over $\tilde k$. It is a proper algebraic stack if $\fX$ is proper.

Furthermore, by gluing the marked point with label $(l)_{ij}^{i'j'}$ to the marked point with label $(l)^{ij}_{i'j'}$ for every $i,j,i',j',l$, we obtain a finite surjective morphism from the stack $\bcMgn(\fXs,A)_\text{dec}$  to the stack $\bcMgn(\fXs,A)$.
Theorem \ref{thm:algebraicity_of_stack_of_stable_maps} implies that the stack $\bcMgn(\fXs,A)$ is an algebraic stack locally of finite presentation over the residue field $\tilde k$.
We deduce that the stack $\bcMgn(\fXs, A)$ is an algebraic stack of finite type over $\tilde k$. It is a proper algebraic stack if $\fX$ is proper.
\end{proof}

\begin{cor}\label{cor:properness_of_formal_moduli}
The moduli stack $\bcMgn(\fX,A)$ is a formal stack finitely presented over $k^\circ$. It is proper if $\fX$ is proper.
\end{cor}
\begin{proof}
It follows from Propositions \ref{prop:moduli_of_formal_stable_maps} and \ref{prop:properness_of_special_fiber}.
\end{proof}

\begin{thm}[Non-archimedean Gromov compactness]\label{thm:non-archimedean_Gromov}
Let $X$ be a $k$-analytic space, and let $\hL$ be a Kähler structure on $X$ with respect to an SNC formal model $\fX$ of $X$.
Then the moduli stack $\bcMgn(X,A)$ is a compact $k$-analytic stack.
If we assume moreover that the $k$-analytic space $X$ is proper and that the residue field $\tilde k$ has characteristic zero, then $\bcMgn(X,A)$ is a proper $k$-analytic stack.
\end{thm}
\begin{proof}
The first part of the theorem follows from \cref{thm:moduli_of_analytic_stable_maps} and the first part of \cref{cor:properness_of_formal_moduli}.
For the second part of the theorem, we note that the SNC formal model $\fX$ is proper if and only if the $k$-analytic space $X$ is proper (cf.\ \cite{Temkin_Local_2000} Corollary 4.4).
So by \cref{cor:properness_of_formal_moduli}, the moduli stack $\bcMgn(\fX,A)$ is a proper formal stack finitely presented over $k^\circ$.
Since the residue field $\tilde k$ is assumed to have characteristic zero,
an \gn formal stable map into $\fX$ has no infinitesimal automorphisms.
Therefore, the diagonal of the formal stack $\bcMgn(\fX,A)$ is unramified.
Thus the second part of the theorem follows from \cref{thm:properness_of_analytic_stacks} and \cref{thm:moduli_of_analytic_stable_maps}.
\end{proof}